\documentclass[12pt]{article}
\usepackage[margin=1.1in]{geometry}
\usepackage{amsmath,amssymb,amsthm, mathrsfs, multirow, float, tikz, mathtools, hyperref, booktabs, chngcntr}
\usetikzlibrary{arrows.meta}
\usepackage[ruled, vlined,linesnumbered]{algorithm2e}
\usepackage{tikz-network}
\usepackage{fancyhdr}
\usepackage[T1]{fontenc}
\usepackage[utf8]{inputenc}
\usepackage[tiny]{titlesec}
\usepackage[numbers]{natbib}

\newtheorem{proposition}{\scshape Proposition}

\newtheorem{definition}{\scshape Definition}
\newtheorem{lemma}{\scshape Lemma}
\newtheorem{exa}{\scshape Example}
\newtheorem{excont}{\scshape Example}

\let\vec\mathbf

\newcommand{\conv}{\operatorname{conv}}

\newtheorem{case}{\scshape Case}
\counterwithin*{case}{subsection} 

\providecommand{\keywords}[1]
{
  \small	
  \textbf{\textit{Keywords---}} #1
}

\title{\vspace{-1cm}\large{\bf{A Polyhedral Approach to Least Cost Influence Maximization in Social Networks}}}
\author{\normalsize{Cheng-Lung Chen} \vspace{-0.1cm}  \\ \footnotesize{Industrial Engineering \& Management Systems, University of Central Florida, USA.} \vspace{-0.2cm}\\
\footnotesize{\texttt{chenglung@knights.ucf.edu}}\\
\normalsize{Eduardo L. Pasiliao} \vspace{-0.1cm} \\ \footnotesize{Air Force Research Laboratory, Eglin AFB, USA.} \vspace{-0.2cm} \\
\footnotesize{\texttt{eduardo.pasiliao@us.af.mil}}\\
\normalsize{Vladimir Boginski} \vspace{-0.1cm} \\ \footnotesize{Industrial Engineering \& Management Systems, University of Central Florida, USA.} \vspace{-0.2cm} \\
\footnotesize{\texttt{vladimir.boginski@ucf.edu}}
} 

\date{}

\begin{document}
\maketitle

\begin{abstract}
\noindent The least cost influence maximization problem aims to determine minimum cost of partial (e.g., monetary) incentives initially given to the influential spreaders on a social network, so that these early adopters exert influence toward their neighbors and prompt influence propagation to reach a desired penetration rate by the end of cascading processes. We first conduct polyhedral analysis on a substructure that describes influence propagation assuming influence weights are unequal, linear and additively separable. Two classes of facet-defining inequalities based on a mixed 0-1 knapsack set contained in this substructure are proposed. We characterize another exponential class of valid and facet-defining inequalities utilizing the concept of minimum influencing subset. We show that these inequalities can be separated in polynomial time efficiently. Furthermore, a polynomial-time dynamic programming recursion is presented to solve this problem on a simple cycle graph. For arbitrary graphs, we propose a new exponential class of valid inequalities that dominates the cycle elimination constraints and an efficient separation algorithm for them. A compact convex hull description for a special case is presented. We illustrate the effectiveness of these inequalities via a delayed cut generation algorithm in the computational experiments.
\end{abstract}


\section{Introduction}
Intricate connections between entities in many natural and man-made systems form large complex networks. Of particular interest in network science is to gain insight into the dynamic of cascading processes in complex networks. For instance, the spreading of new behaviors, opinions, technologies, conventions, and gossips from person to person through a social network may play an important role in designing competitive marketing strategies. Indeed, the collection of social ties among consumers can be exploited to select pivotal early adopters for initiation and anticipate the time and cost required for the information propagation. Therefore, a key research question has been focused on how to efficiently identify a set of users to disseminate a certain information within the network, result in an increasing trend in studying social influence and information propagation in social networks (see, e.g., \cite{chen2013information, kempe2015maximizing}). As it may be expected, the spreading of social influence is commonly studied on network graphs that consist of nodes and links representing users and their connections, as well as a framework that describes how the information propagates among various intermediate users over time. Granovetter \cite{granovetter1978threshold} first presented the threshold model to simulate the collective force exerted by a group on each of its members for predicting innovation and rumor diffusion, voting trends, and migration.  A distinct threshold value is assigned to each user representing the proportion of neighbors who make a decision before a particular user makes such decision. Many extensions built on the threshold model have been proposed to encompass different circumstances. Among them, the linear threshold model and independent cascade model are the two most popular and well-studied one. Kempe et al.  \cite{kempe2003maximizing} study the linear threshold and independent cascade models under the influence maximization (IM) problem. The goal of IM is to activate as many nodes as possible by the end of propagation process given selected influential nodes within budget. Specifically, they show that the expected influence spread is a monotone submodular function, which can be approximated by greedy algorithm with performance guarantee $1-\frac{1}{e}-\epsilon$. On the other hand, the target set selection (TSS) introduced by \cite{chen2009approximability} aims to find the minimum number of users required initially and activate the entire network through the propagation process. These two types of problems are fundamental in social network analysis and have many variants. Here we refer the reader to 
a survey of hardness and solution approaches for TSS \cite{banerjee2020survey}, a review on approximation algorithms for IM \cite{li2018influence}, and a comprehensive review on introduction to types of social networks, properties, evaluation metrics and known method to solve IM \cite{peng2018influence}. Recently, Azaouzi et al. \cite{azaouzi2021new} presented a comprehensive review on models and methods of group-level IM and IM under privacy protection. 

In this paper, we study an optimization problem arising in social network analytics, referred to as the \emph{least cost influence maximization (LCIM)} problem \cite{fischetti2018least, gunnecc2020branch, gunnecc2020least}. The propagation process in LCIM is based on the linear threshold model, where each user is assigned a real-valued threshold and each link between users is assigned a weight to represent influence level. If the summation of influence weights from neighbors exceeds one’s threshold value, then the individual is \textit{activated}, meaning that the information is adopted or the person changes their opinion to align with friends (neighbor nodes) in the social network. In LCIM, the concept of influence weights is extended with the idea of giving partial incentives such as free samples or coupons to individual as a motivation for spreading information. The assumption of the linear threshold model allows the incentives and influence weights in the activation processes to be linearly additive, which immediately admits a mixed-integer optimization formulation for the influence propagation. Our goal is to develop an exact computational method based on the polyhedral structure. We assume that all the parameters are deterministic, all the nodes are inactive initially and nodes remain activated once the influence weights exceeds the threshold. A similar assumption on using deterministic parameters in linear threshold model can be found in \cite{gursoy2018influence}. From a practical point of view, the assumption of deterministic linear threshold depends on the accuracy of estimation of influence factors and threshold parameters. Machine learning and data mining techniques may enable one to obtain accurate predictions on those parameters from massive amounts of data available nowadays. Note that when incentive is allowed and the influence weights and thresholds are considered known in the problem, the influence propagation is no longer submodular, thus, a greedy algorithm is not applicable for arbitrary graphs. However, the problem is still challenging to solve due to the combinatorial nature. The minimum subset of nodes required under dynamic activation is also referred to as the dynamic monopoly, where nodes become activated if at least half of the neighbors are activated in the previous time period. Hence, an extra index of time is incorporated into the integer programming formulation to describe step-by-step activation of nodes, see \cite{moazzez2018integer,moazzez2021facets,soltani2019polyhedral}.  We do not consider time dynamics in our problem despite the fact that the influence propagation occurs in discrete steps. Moreover, the induced optimal influence propagation graph for LCIM in the static network setting given by the mixed-integer programming formulation is known to be acyclic \cite{fischetti2018least}. 

Even though IM, TSS and LCIM problems share certain similarities, most of the previous studies on IM or TSS mainly focus on developing degree-based or centrality-based metaheuristics and approximation algorithms. Existing studies on using exact mixed-integer optimization techniques for LCIM under deterministic data settings are relatively limited. The computational complexity of LCIM is established in \cite{gunnecc2020least}. In particular, they show that LCIM is NP-hard on arbitrary graphs and bipartite graphs for both equal and unequal influence when 100\% penetration rate is not required. They also give a greedy algorithm and a total unimodular formulation for LCIM with equal influence on a tree and 100\% penetration rate. G{\"u}nne{\c{c}} et al. \cite{gunnecc2020branch} develop a branch-and-cut algorithm using this total unimodular formulation for LCIM on arbitrary graphs. Fischetti et al. \cite{fischetti2018least} present a novel set covering formulation for a generalized LCIM, where the activation function can be adjusted to be nonlinear in order to capture the situation of diminishing marginal influences or over-proportional effect from peers. They propose strengthened generalized propagation inequalities and a price-cut-and-branch algorithm to deal with the exponential number of variables and constraints. Recent developments on using exact mixed-integer optimization methods for stochastic influence maximization problem include a delayed constraint generation algorithm for a two-stage stochastic influence maximization problem \cite{wu2018two} and a branch-cut-and-price algorithms for robust a influence maximization, where node thresholds and arc influence factors are subject to budget uncertainty \cite{nannicini2020robust}.

\subsection{Notation and problem definition}
For convenience, we use the notation $[a,b]$ to denote the set $\{a,a+1,\ldots,b\}$ for $a\leq b$, and $[a,b] = \varnothing$ for $a > b$. For a set $\mathcal{R} \subseteq \mathbb{R}^n$, we use $\conv{(\mathcal{R})}$ to denote its convex hull of solutions. Formally, an oriented network (e.g., a social network) is represented by a graph $G=(V,E)$ with the set of nodes $V:=\{1,\ldots,n\}$ corresponding to people or users and the set of bidirectional arcs $E \subseteq \{(i,j) \in V \times V: (j,i) \in V \times V, i \neq j\}$ with cardinality $m$ corresponding to connections and influence directions between people in the network. Hence, each node has identical number of predecessors and successors, denoted by $v_i$: $v_i = |N(i)|$, where $N(i) := \{j \in V: (j,i) \in E\}$. Each node $ i \in V$ has a threshold value $h_i$ measuring the ``difficulty level'' of an individual to be ``activated''. 
Each arc $(i,j) \in E$ is associated with an influence weight $d_{ij}$. The coverage (penetration) rate is denoted by $a$, where $ 0 < a \leq 1$ is assumed. We also assume that $d_{ij}$ and $h_i$ are positive integers such that $\max\{d_{ji}: j \in N(i) \} \leq h_i$ and $\sum_{j \in N(i)}d_{ji} > h_i$ for all $i \in V$ such that $|N(i)|\geq 2$ throughout this paper to omit trivial cases. For any $d_{ji} > h_i$ and $|N(i)| = 1$ for some $i \in V$, we pre-process the data to set $d_{ji} = h_i$. All nodes are assumed inactive initially and nodes remain active once influences from neighbors and incentives reach or exceed the threshold. For each node $ i \in V$, let nonnegative continuous variables $x_i$ be the amount of partial incentives given to user $i$, binary variables $y_{ij}$ indicate whether influence is exerted from node $i$ to $j$ nor not, and binary variables $z_i$ indicate whether node $i$ is activated or not. The arc-based formulation of static LCIM is given by
\begin{subequations}
\begin{align}
\text{(LCIM) } \min\limits_{x,y,z} \quad & \sum_{i \in V} x_i  \notag  \\
\text{s.t.  } & x_i + \sum_{j \in N(i)}d_{ji}y_{ji} \geq h_iz_i \quad \forall i \in V \label{LCIM1} \\
& y_{ij} + y_{ji} \leq z_i \quad \forall (i,j) \in E  \label{LCIM2} \\
&   \sum\limits_{i \in V} z_i \geq b \label{LCIM3} \\
& \sum_{(i,j) \in C} y_{ij} \leq \sum_{i \in V(C)\setminus \{k\}} z_i  \quad \forall k \in V(C), \forall \text{ cycles } C \subseteq E \label{LCIM4} \\
& x \in \mathbb{R}_+^n \notag  \\
& y \in \mathbb{B}^m, z \in \mathbb{B}^n. \notag
\end{align}
\end{subequations}
Node propagation constraints \eqref{LCIM1} follow the linear threshold model by evaluating the total incoming influence from neighbor plus the incentives given to a node. Constraints \eqref{LCIM2} state that for every two nodes with bidirectional arcs, the influence exertion is allowed in one way only. The cardinality constraint \eqref{LCIM3} describes the desired number of activated nodes given a predetermined penetration rate $a$, where $b=\lceil an \rceil$ and $1\leq b \leq n$.  Constraints \eqref{LCIM4} are generalized cycle elimination constraints (GCEC) where $V(C) = \{i \in V: (i,j) \in C\}$. They are necessary to produce the acyclic optimal influence propagation graph. Note that the arc-based formulation proposed by \cite{ackerman2010combinatorial} is different from this paper as the influence weights in their model are coming solely from their neighbors without incentives. Fischetti et al. \cite{fischetti2018least} adopt this arc-based formulation on arbitrary graphs for computational performance comparison, but possible values of incentives are discretized by a set of an exponential number of binary variables.

\subsection{Main contributions}
The key contributions of this paper are summarized as follows. We conduct a polyhedral study for the hidden mixed 0-1 knapsack substructure in \eqref{LCIM1} of LCIM. We propose three classes of strong valid inequalities, namely, the continuous cover, continuous packing and the minimum influencing subset inequalities from this substructure and specify the conditions under which they are facet-defining. The coefficient of these inequalities can be adjusted to equal influence assumption directly. We have improved the run time of the separation algorithm for the the continuous cover and continuous packing inequalities compared with our preliminary results presented in \cite{chen2020cutting}. In addition, we provide a polynomial-size complete linear description of the polyhedron of LCIM on a tree when equal influence weights for every node and 100\% coverage are assumed. For LCIM on arbitrary bidirectional network graphs, we derive a novel class of strong valid inequalities called the the $(U,C)$ inequalities and show they dominate the general cycle elimination constraints. Finally, we augment the preliminary computations in \cite{chen2020cutting} with different formulations, cutting plane strategies, density and scale of networks in extensive computational experiments. 

\subsection{Outline}
The remainder of this paper is organized as follows. In Section \ref{sec:2}, we derive strong valid inequalities from the mixed 0-1 knapsack substructure and give an exact polynomial time separation algorithm for these inequalities. In Section \ref{sec:3}, we discuss an $O(n)$ time dynamic programming recursion for LCIM on a simple cycle. We propose the $(U,C)$ inequalities for arbitrary bidirectional graphs and develop a polynomial time separation algorithm via solving the longest path problem on a directed acyclic graph. We present the convex hull description of LCIM on a special case in Section \ref{sec:4}. Finally, we illustrate the effectiveness of our proposed valid inequalities in the computational experiments in Section \ref{sec:5} and conclude with Section \ref{sec:6}. 

\section{Valid inequalities based on mixed 0-1 knapsack polyhedron} \label{sec:2}
To develop a strong formulation for LCIM, we study the polyhedral structure of constraints \eqref{LCIM1}.  For $i \in V$, let
\begin{align}
     \mathcal{P}_i = \left\{(x_i,y,z_i) \in \mathbb{R}_+ \times \mathbb{B}^{v_i+1}: x_i + \sum_{j \in N(i)}d_{ji} y_{ji} \geq h_iz_i\right\}.  \notag
\end{align}
The set $\mathcal{P}_i$ is a mixing set with a binary variable on the right-hand side value. For any inequality that is facet-defining for $\conv{(\mathcal{P}_i)}$, it is facet-defining for $\conv{(\cap_{i\in V} \mathcal{P}_i)}$ as well. Therefore, we now consider a single node propagation by dropping the subscript $i$ and obtain the following set
\begin{align}
     \mathcal{P} = \left\{(x,y,z) \in \mathbb{R}_+ \times \mathbb{B}^{v+1} : x + \sum_{j \in N}d_j y_j \geq hz\right\}. \notag
\end{align}
Observe that the set $\mathcal{P}$ contains a mixed 0-1 knapsack structure. Let set $\mathcal{\overline{P}}$ obtained from $\mathcal{P}$ by setting $\overline{y}_j = 1 - y_j$, $j \in N$ and $z=1$, then we have a mixed 0-1 knapsack set $\overline{\mathcal{P}}$ with weight $d_j$ for each item $j \in N$ and the capacity of knapsack $\sum_{j \in N}d_j - h $ plus an unbounded continuous variable $x$ in the following
\begin{align}
      \overline{\mathcal{P}} = \left\{(x,\overline{y},z) \in \mathbb{R}_+ \times \mathbb{B}^v \times \{1\}: \sum_{j \in N}d_j \overline{y}_j \leq \left(\sum_{j \in N}d_j - hz\right) + x \right\}. \notag
\end{align}

The mixed 0-1 knapsack set $\overline{\mathcal{P}}$ is a special case of traditional 0-1 knapsack problem where the knapsack size is expanded with additional capacity. Observe that $\dim(\mathcal{P})$ is full-dimensional and contains the origin. There are trivial facets for $\conv{(\mathcal{P})}$ and their verification is straightforward. 

\begin{proposition} The following facet-defining inequalities for $\conv{(\mathcal{P})}$ are trivial. \label{prop:trivial}
\begin{itemize}
    \item[(i)] The inequality $x + \sum_{j \in N}d_j y_j \geq hz$ is facet-defining for $\conv{(\mathcal{P})}$ if $d_j \leq h$ for all $j \in N$. 
    \item[(ii)] The inequality $x \geq 0$ is facet-defining for $\conv{(\mathcal{P})}$ if $d_j \leq h$ for all $j \in N$. 
    \item[(iii)] The inequality $y_j \geq 0$ is facet-defining for $\conv{(\mathcal{P})}$.
    \item[(iv)] The inequality $y_j \leq 1$ is facet-defining for $\conv{(\mathcal{P})}$.
\end{itemize}
\end{proposition}

 The first result of Proposition \ref{prop:trivial} validates our data pre-processing assumption for influence weight and threshold when $d_{ij} > h_i$. To show this, consider a single node influence propagation inequality 
 \begin{align}
x + \sum_{j \in N: d_j \leq h}d_j y_j + \sum_{j \in N: d_j > h}d_j y_j \geq hz. \notag 
 \end{align}
 Clearly, this inequality is dominated by 
 \begin{align}
x + \sum_{j \in N: d_j \leq h}d_j y_j + \sum_{j \in N: d_j > h}h y_j \geq hz \notag 
 \end{align}
 as the second inequality is faced-defining. Hence, we can obtain the facet-defining inequality by simply substituting any influence weight $d_{ji}$ with $h_i$ for all $j \in N(i)$ if $d_{ji} > h_i$. Marchand and Wolsey \cite{marchand19990} propose two classes of valid inequalities for $\overline{\mathcal{P}}$ based on mixed-integer rounding and lifting function, namely, the continuous cover and continuous packing (reversed cover) inequalities.  We follow a similar idea to strengthen the formulation for LCIM with moderate modification as $\overline{\mathcal{P}} \subset \mathcal{P}$. Applications of continuous cover and continuous packing inequalities for $\mathcal{\overline{P}}$ can be seen in several fields, including delay management for public transportation \cite{dal2019strengthened}, job scheduling with uncertain multiple resources \cite{keller2009scheduling}, discrete lot sizing \cite{loparic2003dynamic} and single-item capacitated lot sizing \cite{miller2000capacitated}.  They can also be extended to solve general mixed-integer optimization that contains mixed 0-1 knapsack set with bounded continuous variables, see \cite{narisetty2011lifted,richard2003lifteda,richard2003liftedb}.

\subsection{Continuous cover and continuous packing inequalities}
Consider a continuous cover $S := \{1,\ldots,s\} \subseteq N$ such that $h + \sum_{j \in S}d_j - \sum_{j \in N}d_j = \pi > 0$ and  $h + \sum_{j \in S \setminus \{k\}}d_j - \sum_{j \in N}d_j < 0$ for any $k \in S$. Let $d_j \in S$ be in non-increasing order with $d_1 \geq \ldots \geq d_r > \pi \geq d_{r+1} \geq \ldots \geq d_s$, $D_j = \sum_{k=1}^j d_k$ for $j \leq r$ and $D_0 = 0$. 
\begin{proposition}
The continuous cover inequality
\begin{align}
    &x + \sum_{j \in S} \min\{\pi, d_j\}y_j + \sum_{j \in N \setminus S}\Phi(d_j)y_j \geq  \left(\min\limits_{j\in S} \{\pi, d_j\} + \sum_{j \in N \setminus S}\Phi(d_j) \right)z  \label{cci}
\end{align}
where 
\begin{align}
\Phi(d) = 
    \begin{cases}
     j\pi &  D_j \leq d \leq D_{j+1} - \pi, \quad j \in [0,r-1] \\
     j\pi + d - D_j & D_j - \pi \leq d \leq D_j, \quad j \in [1,r-1] \label{phiS}  \\
     r\pi + d - D_r  & D_r - \pi \leq d,
    \end{cases}
\end{align}
is valid for $\mathcal{P}$. 
\end{proposition}

Similarly, consider a continuous packing $L := \{1,\ldots,l\} \subseteq N$ such that $\sum_{j \in L}d_j - h = \lambda > 0$ and  $\sum_{j \in L \setminus \{k\}}d_j - h  < 0$ for any $k \in L$. Let $d_j \in L$ be in non-increasing order with $d_1 \geq \ldots \geq d_r > \lambda \geq d_{r+1} \geq \ldots \geq d_l$, $D_j = \sum_{k=1}^j d_k$ for $j \leq r$ and $D_0 = 0$. 

\begin{proposition}
The continuous packing inequality
\begin{align}
x + \sum_{j \in L} \max\{0, d_j-\lambda\}y_j + \sum_{j \in N \setminus L}\Psi(d_j)y_j \geq \left(\sum_{j \in L} \max\{0, d_j-\lambda\}\right)z  \label{cpi}   
\end{align}
where
\begin{align}
\Psi(d) = 
    \begin{cases}
    d - j\lambda &  D_j \leq d \leq D_{j+1} - \lambda, \quad j \in [0, r-1] \\
    D_j - j\lambda & D_j - \lambda \leq d \leq D_j, \quad j \in [1,r-1] \label{psiT}  \\
    D_r - \lambda r & D_r - \lambda \leq d,
    \end{cases}
\end{align}
is valid for $\mathcal{P}$. 
\end{proposition}

Here we omit the proofs as the validity of inequalities \eqref{cci} and \eqref{cpi} and their facet-defining conditions for $\conv{(\mathcal{P})}$ directly follow from \cite{marchand19990} when $z=1$, while when $z=0$, we must have $y_j = 0$ for all $j \in N$, both inequalities are trivially satisfied and facet-defining according to Proposition \ref{prop:trivial}. 

\begin{exa}\label{ex1}
 \normalfont Let $d = (7,6,5,4)$ and $h=8$, we list the facet-defining inequalities of inequality \eqref{cci} and \eqref{cpi} in Table \ref{tab:ex1}. For example, for $S=\{1,2,4\}$, we have $\pi = 3$ and $r = 3$. Then the lifting function $\Phi$ is given by
\begin{align*}
\Phi(d) = 
    \begin{cases}
    0 & 0 \leq d \leq 4 \\
    3 & 7 \leq d \leq 10 \\
    6 & 13 \leq d \leq 14 \\
    d-4 & 4 \leq d \leq 7 \\
    d-7 & 10 \leq d \leq 13 \\
    d-8 &  14 \leq d 
    \end{cases}   
\end{align*}
Hence, the coefficient of $y_3$ is $\Phi(d_3) = \Phi(5) = 5 - 4 =1$, which generates 
\begin{align}
x +3y_1 + 3y_2 + y_3 + 3y_4 \geq 4z. \notag    
\end{align}

\begin{table}[H]
\caption{Continuous cover and continuous packing inequalities of Example \ref{ex1}}
\centering
\begin{tabular}{ccc}
\toprule
\multicolumn{2}{c}{$x + 7y_1 + 6y_2 + 5y_3 + 4y_4 \geq 8z$} \\ \hline
 set & facet-defining inequality \\ \hline
$S = \{2,3,4\}$, $L = \{1,2\}$ & $x + y_1 + y_2 + y_3 + y_4 \geq 2z$ \\
$S = \{1,3,4\}$, $L = \{1,2\}$ & $x +2 y_1 + y_2 + 2y_3 + 2y_4 \geq 3z$ \\
$S = \{1,2,4\}$, $L = \{1,3\}$ & $x +3y_1 + 3y_2 + y_3 + 3y_4 \geq 4z$ \\
$S = \{1,2,3\}$, $L = \{1,4\}$ & $x +4y_1 + 4y_2 + 4y_3 + y_4 \geq 5z$ \\
$S = \{1,2,4\}$, $L = \{2,3\}$ & $x +4y_1 + 3y_2 + 2y_3 + 3y_4 \geq 5z$ \\
$S = \{1,2,3\}$, $L = \{2,4\}$ & $x +5y_1 + 4y_2 + 4y_3 + 2y_4 \geq 6z$ \\
$S = \{1,2,3\}$, $L = \{3,4\}$ & $x +6y_1 + 5y_2 + 4y_3 + 3y_4 \geq 7z$ \\ \bottomrule
\end{tabular}
\label{tab:ex1}
\end{table}
\end{exa}

Essentially, the continuous cover inequalities \eqref{cci} and packing inequalities \eqref{cpi} are not sufficient to describe $\conv{(\mathcal{P})}$, as the additional binary variable $z$ creates new extreme points. Next, we introduce a new class of valid inequalities for $\mathcal{P}$ that utilizes the concept of minimal influencing set.

\subsection{Minimum influencing subset inequalities}
We use the definition of minimal influencing set from \cite{fischetti2018least}, which we include here for the reader's convenience.
\begin{definition}[\cite{fischetti2018least}]
Let $p_i \in [0,h_i-1]$ be an incentive payment to node $i \in V$ and $M \subseteq N(i)$ be a set of active neighbors of node $i \in V$, such that  $p_i + \sum_{j \in M}d_{ji} = h_i$. We say $M$ is a minimal influencing subset for node $i \in V$ if and only if for a fixed incentive payment $\overline{p}_i$, it satisfies $\overline{p}_i + \sum_{j \in M}d_{ji} = h_i$ and $\overline{p}_i + \sum_{j \in M \setminus\{k\} }d_{ji} < h_i$ for any $k \in M$. In other words, a strict subset of $M$ with the same incentive payment is not sufficient to activate node $i$. 
\label{def1}
\end{definition}

\begin{proposition}
Let $M \subseteq N$ be a minimum influencing subset with an incentive payment $p = h - \sum_{j \in M}d_j $. The minimal influencing subset inequality
\begin{align}
    x + \sum_{j \in N\setminus M} \min \{d_j, p\} y_j \geq pz \label{mis} 
\end{align}
is valid for $\mathcal{P}$.
\end{proposition}

\begin{proof}
If $z=0$ then inequality \eqref{mis} is trivially satisfied. If $y_j=0$ for all $j \in N\setminus M$, either $x = 0$ for $z=0$ or $x = p$ for $z=1$. Assume that none of these cases hold, given a $p > 0$, rewrite the left term of the inequality in the following form:
\begin{align}
    & x + \sum_{j \in N}d_jy_j  \notag \\ 
    = ~ & x + \sum_{j \in N\setminus M: d_j \leq p}d_jy_j + p \sum_{j \in N\setminus M: d_j > p}y_j  + \sum_{j \in M}d_j y_j \geq h, \notag
\end{align}
which implies
\begin{align}
   x + \sum_{j \in N\setminus M: d_j \leq p}d_jy_j + p \sum_{j \in N\setminus M: d_j > p}y_j  \geq h - \sum_{j \in M}d_j y_j \geq  h - \sum_{j \in M}d_j = p. \notag
\end{align}
\end{proof}

\begin{proposition}
Inequality \eqref{mis} is facet-defining for $\conv{(\mathcal{P})}$ if and only if $p>0$. Moreover, for a given $i \in V$ and a set $N(i)$, for each $M \subseteq N(i)$ such that $h_i - \sum_{j\in M}d_{ji} = p_i > 0$, the minimal influencing subset inequality
\begin{align}
    x_i + \sum_{j \in N(i)\setminus M} \min \{d_{ji}, p_i\} y_{ji} \geq p_i z_i \label{mis-LCIM}
\end{align}
is facet-defining for $\conv{(\cap_{i\in V}\mathcal{P}_i)}$. 
\end{proposition}

\begin{proof}
Recall that $p \in [0, h-1]$ by definition. If $p= 0$, then inequality \eqref{mis} reduces to $x \geq 0$; therefore, $p>0$ is necessary. To show the sufficiency that inequality \eqref{mis} is facet-defining for $\conv{(\mathcal{P})}$,  we exhibit $v+1$ linearly independent points $(x,\vec{y},z)$ on the face defined by inequality \eqref{mis}. Let $e_j \in \mathbb{B}^v$ be the unit vector corresponding to $y_j$ for $j \in N$. Consider the two feasible points $(p, \sum_{j \in M}e_j, 1)$ and $(0, \sum_{j \in M}e_j, 0)$, then, consider the $v-1$ feasible points $(0, \sum_{j \in M}(e_j + e_k), 1)$ for $k \in L \setminus M$. It is straightforward to verify that these $v+1$ points are linearly independent and satisfy inequality \eqref{mis} at equality.  To prove the second part of this proposition, let $\eta_i \in \mathbb{B}^{2n+m}, \mu_{ij} \in \mathbb{B}^{2n+m}$, and $\zeta_i \in \mathbb{B}^{2n+m}$ for $i \in V, j \in N(i)$ be the unit vectors corresponding to variables $x_i, y_{ij}$, and $z_i$, respectively.  The component of $\eta_i$ is 1 if it has the same position with $x_i$ in the feasible solution; all other components of $\eta_i$ are 0. Similar setting is made to $\mu_{ij}$ for $y_{ij}$ and $\zeta_i$ for $z_i$, respectively. For $i \in V$, consider the $n$ points $p_i\eta_i + \sum_{j \in M}\mu_{ji} + \zeta_i$, also, consider the $n$ points $\sum_{j \in M}\mu_{ji}$. For $ i \in V$ and $k \in L \setminus M$, consider the $m-1$ points $\sum_{j \in M}(\mu_{ji} + \mu_{ki}) + \zeta_i$. These $2n+m-1$ points are linearly independent and satisfy inequality \eqref{mis-LCIM} at equality, which completes the proof. 
\end{proof}

\begin{excont}[Continued]
 \normalfont The facet-defining inequalities of \eqref{mis} for Example \ref{ex1} are listed in Table \ref{tab:ex1-c}.
\begin{table}[H]
\caption{Minimal influencing subset inequalities of Example \ref{ex1}}
\centering
\begin{tabular}{cc}
\toprule
\multicolumn{2}{c}{$x + 7y_1 + 6y_2 + 5y_3 + 4y_4 \geq 8z$} \\ \hline
set & facet-defining inequality \\ \hline
$M = \{1\}$ & $x  +y_2 + y_3 + y_4 \geq z$ \\
$M = \{2\}$ & $x +2 y_1  +2y_3 + 2y_4 \geq 2z$ \\
$M = \{3\}$ & $x +3y_1 + 3y_2 + 3y_4 \geq 3z$ \\
$M = \{4\}$ & $x +4y_1 + 4y_2 + 4y_3 \geq 4z$ \\ \bottomrule
\end{tabular}
\label{tab:ex1-c}
\end{table}
\end{excont}

Although inequalities \eqref{cci}, \eqref{cpi} and \eqref{mis} define a large number of facets for $\conv{(\mathcal{P})}$, they are not sufficient to completely describe $\conv{(\mathcal{P})}$ in its original space of variables. Particularly, the following inequality is valid and facet-defining for Example \ref{ex1} but can not be obtained through either inequalities \eqref{cci}, \eqref{cpi} or \eqref{mis}:
\begin{align*}
    x + 3y_1 + 2y_2 + 2y_3 + 2y_4 \geq 4z.
\end{align*}

\subsection{Separation of minimal influencing subset inequalities}
In this section, we give an exact polynomial time separation algorithm for finding the most violated minimal influencing subset inequality. From inequality \eqref{mis-LCIM}, we observe that finding the most violated inequality for a given fractional solution $(\vec{x^*},\vec{y^*},\vec{z^*}) \in \mathbb{R}_+^{2n+m}$ consists of choosing a set $M \subseteq N(i)$ such that $p_i z_i -  \sum_{j \in N(i)\setminus M} \min \{d_{ji},p_i\} y_{ji}$ is maximized. Let $\hat{v}:= \max \{v_i: i \in V\}$.

\begin{proposition}\label{prop:sep_M}
Given a fractional solution $(\vec{x^*},\vec{y^*},\vec{z^*}) \in \mathbb{R}_+^{2n+m}$ from solving LCIM, there exists an $O(n\hat{v}\log \hat{v})$ time separation algorithm for inequality \eqref{mis-LCIM}. 
\end{proposition}

\begin{proof}
A violated inequality \eqref{mis-LCIM} can be found if  
\begin{align}
p_i \left(z_i^*  -  \sum_{j \in N(i)\setminus M: d_{ji} > p_i}y_{ji}^* \right) - \sum_{j \in N(i)\setminus M: d_{ji} \leq p_i}d_{ji}y_{ji}^* > x_i^*, \notag
\end{align}
which implies that it suffices to consider $y_{ji}^*$ for some $j \in N(i)$ such that $z_i^* - \sum_{j \in N(i)}y_{ji}^* > 0$ and $p_i > 0$. To do so, we sort $y_{ji}^*$ in non-decreasing order for $j \in N(i)$ with indices $j_1, j_2, \ldots, j_v$ such that $y_{j_1 i}^* \leq y_{j_2 i}^* \leq \ldots \leq y_{j_v i}^*$. For $j_1\leq j_k \leq j_v$, we sum up first $k$ elements, then we check if $z_i^* - \sum_{\ell=1}^k y_{j_{\ell}i}^* >0$ and $ p_i^\prime = h_i - \sum_{\ell=k+1}^v d_{j_{\ell}i} > 0$,  until $z_i^* - \sum_{\ell=1}^{k+1} y_{j_{\ell}i}^* <0$. These $k$ elements constitute the subset $M$ and $N(i)\setminus M$ simultaneously and ensure $z_i^* - \sum_{j \in N(i)\setminus M}y_{ji}^* > 0$ and $p_i > 0$ in order to generate a violated cut. The set $M$ corresponds to the most violated cut can be determined by evaluating $\max\{0, p_i^\prime(z_i^* - \sum_{\ell=1}^k y_{j_{\ell}i}^*): k \in [1,v]\}$. If $\max\{0, p_i^\prime(z_i^* - \sum_{\ell=1}^k y_{j_{\ell}i}^*): k \in [1,v]\} = 0$, then there are no violated cuts. The sorting process runs in $O(\hat{v} \log \hat{v})$ time and the evaluation takes  $O(\hat{v})$, since we have to check for every node $i \in V$, overall the separation algorithm runs in  $O(n\hat{v}\log \hat{v})$ time.   
\end{proof}

\subsection{Separation of continuous cover and continuous packing inequalities}
Up to this point, we presented an exact polynomial time separation algorithm for inequalities \eqref{mis-LCIM}. Next, we show that a violated continuous cover inequality for $\conv{(\cap_{i\in V}\mathcal{P}_i)}$ can be identified by exploiting the result of Proposition \ref{prop:sep_M}. Then we can obtain a violated continuous packing inequality in polynomial time after finding a violated continuous cover inequality. 
We first formally establish the relationship between $S$, $L$ and $M$ in the following lemma. 

\begin{lemma} \label{lemma1}
If $p = h - \sum_{j \in M}d_j > 0$ and there exists $k \in N\setminus M$ such that $\sum_{j \in M \cup \{k\}}d_j > h$ and $\sum_{j \in M \cup \{k\}\setminus \{\ell\}}d_j < h$ for any $\ell \in M$ , then $p = \pi$, $S = N \setminus M$, $\sum_{j \in M \cup \{k\}}d_j - h = \lambda$ and $L = M\cup\{k\}$.
\end{lemma}

\begin{proof}
By rearranging the terms in the definition of $p$,
\begin{align}
 p & =  h - \sum_{j \in M}d_j > 0 \notag \\
   & =  h + \sum_{j \in N\setminus M}d_j - \sum_{j \in N}d_j > 0 \notag \\ 
   & =  h + \sum_{j \in S}d_j - \sum_{j \in N}d_j > 0.  \notag
\end{align}
Hence, we can see that $S$ is equivalent to $N\setminus M$ and $p=\pi$ if $p>0$. Next, if there exists an element $k \in N\setminus M$ such that $\sum_{j \in M \cup \{k\}}d_j > h$,  immediately we have $M \cup \{k\} = L$ by definition and $\sum_{j \in M \cup \{k\}}d_j - h = \lambda$. 
Since $N\setminus M \cup \{M\cup\{k\}\} = N$, we have $S \cup L = N$, $S \cap L = \{k\}$, and  $L\setminus \{k\} = N\setminus S = M$. 
\end{proof}

Following Lemma \ref{lemma1}, we present an efficient separation procedure to determine violated continuous cover and continuous packing inequalities by using the information of the set $M$. Note that here we add an index $i$ to inequalities \eqref{cci} similar to \eqref{mis-LCIM} for all $i \in V$. 

\begin{proposition}
There exists a violated continuous cover inequality if a violated inequality \eqref{mis-LCIM} is identified. 
\end{proposition}

\begin{proof}
Recall that inequality \eqref{mis-LCIM} is violated if 
\begin{align}
p_i \left(z_i^*  -  \sum_{j \in N(i)\setminus M: d_{ji} > p_i}y_{ji}^* \right) - \sum_{j \in N(i)\setminus M: d_{ji} \leq p_i}d_{ji} y_{ji}^* > x_i^*, \notag
\end{align}
or equivalently by Lemma \ref{lemma1}, 
\begin{align}
\pi_i z_i^*  - \pi_i \sum_{j \in S: d_{ji} > \pi_i}y_{ji}^*  - \sum_{j \in S: d_{ji} \leq \pi_i}d_{ji} y_{ji}^* > x_i^*. \notag  
\end{align}
Now, a continuous cover inequality for a fixed node $i \in V$ is violated if
\begin{align}
\min\limits_{j \in S}\{\pi_i, d_{ji}\}z_i^* + \sum_{j \in N(i) \setminus S}\Phi(d_{ji})(z_i^*-y_{ji}^*) -  \sum_{j \in S} \min\{\pi_i, d_{ji}\}y_{ji}^* > x_i^*. \notag
\end{align}
Suppose $d_{ji} > \pi_i $ for all $j \in S$, then the left term of the continuous cover inequality can be further written as
\begin{align}
\pi_i z_i^* + \sum_{j \in N(i) \setminus S}\Phi(d_{ji})(z_i^*-y_{ji}^*) - \pi_i \sum_{j \in S: d_{ji} > \pi_i}y_{ji}^*  - \sum_{j \in S: d_{ji} \leq \pi_i}d_{ji} y_{ji}^*. \notag
\end{align}
Since $(z_i^*-y_{ji}^*) \geq 0$ holds and the lifting function $\Phi$ is nonnegative, the rest of the terms already violate the current solution $(\vec{x^*},\vec{y^*},\vec{z^*})$, we then obtain a violated continuous cover inequality with $\pi_i$ being the minimum among $\min\{\pi_i, d_{ji}: j \in S\}$. This suggests that, when a violated inequality \eqref{mis-LCIM} is found, it suffices to generate a violated continuous cover inequality concurrently. 
\end{proof}

On the other hand, a violated continuous packing inequality can not be obtained directly from separating inequality \eqref{mis-LCIM}. However, when a violated continuous cover inequality is found, we can check every element in $S$ to see if there exists an element $k$ such that $\sum_{j \in N\setminus S \cup \{k\}}d_{ji} > h_i$. For every $k$ satisfying this condition, the packing set $L$ is then determined. A violated continuous packing inequality is found if 
\begin{align}
    \sum_{j\in L}\max\{0,d_{ji}-\lambda_i\}(z_i^*-y_{ji}^*) - \sum_{j\in N(i)\setminus L}\Psi(d_{ji})y_{ji}^* > x^* \notag
\end{align}
holds. Furthermore, suppose $\hat{S} = \max\{|S|: S \subseteq N(i), i \in V \}$, the process of checking elements in $S$ takes $O(\hat{S})$ time  and the function $\Psi$ can be constructed in  $O(\hat{v}\log \hat{v})$ time using binary search proposed in \cite{narisetty2011lifted}.

\section{Valid inequalities for LCIM with cycles} \label{sec:3}
In this section, we expand the study of $\conv{(\cap_{i\in V} \mathcal{P}_i)}$ to incorporate the remaining constraints in LCIM on an arbitrary bidirectional graph that contains cycles. The polyhedron that describes the intersection of these constraints is
\begin{align}
     \mathcal{Q} = \left\{(x,y,z) \in \mathbb{R}_+^n \times \mathbb{B}^{n+m}: \eqref{LCIM1} - \eqref{LCIM4} \right\}.  \notag
\end{align}

To simplify the notation, we let $\boldsymbol\alpha$ and $\boldsymbol\beta$ be the coefficients associated with variables $\vec{y}$ and $\vec{z}$ corresponding to continuous cover and continuous packing inequalities. In other words, we express them in the following form:
\begin{align}
    x_i + \sum_{j \in N(i)}\alpha_{ji} y_{ji} \geq \beta_i z_i, \label{ucbase}
\end{align}
and it is facet-defining for $\conv{(\mathcal{Q})}$. Due to the straightforward structure of a cycle, the minimum incentive of the influence propagation can be easily characterized. Raghavan and Zhang \cite{raghavan2021weighted} give an $O(n)$ time algorithm for the weighted target set selection problem on a cycle. We also present a polynomial time dynamic programming to solve LCIM on a simple cycle. Then we give a class of exponential number of valid inequalities that forms acyclic influence propagation by exploiting inequality \eqref{ucbase} as the base inequality. We also demonstrate that the separation for this class of valid inequalities can be done in polynomial time for arbitrary bidirectional graphs that contain cycles. 

\subsection{Dynamic programming recursion for LCIM on a simple cycle} \label{sec:dp}
Without loss of generality, we assume $|V(C)|=|V|=n$ in this section. Hence, for a smple cycle graph, we have $V=V(C)=n$, $E = C$ and $|E|=2n$ with $v_i =2 $ for all $i \in V$. We still use $V(C)$ for the set of all nodes and $C$ for set of all bidirectional arcs to focus on the discussion of LCIM on a simple cycle for consistency. Observe that due to the cycle structure, the influence propagation occurs on an induced path of a cycle as it is one-way and consecutive after a particular node is activated by paying full incentive to it. Given the cardinality requirement $b \leq n-1$, the cost of activating other nodes on this path is equivalent to the threshold of the inactivated node minus the influence weights exerted from the activated predecessor on the side. For $b=n$, the cost of activating the last node is zero as it receives influence exertion from both its predecessor and the firstly activated node. Therefore, we only need to evaluate the cases for $b\leq n-1$. 

\begin{figure}[H]
    \centering
    \includegraphics[scale=0.4]{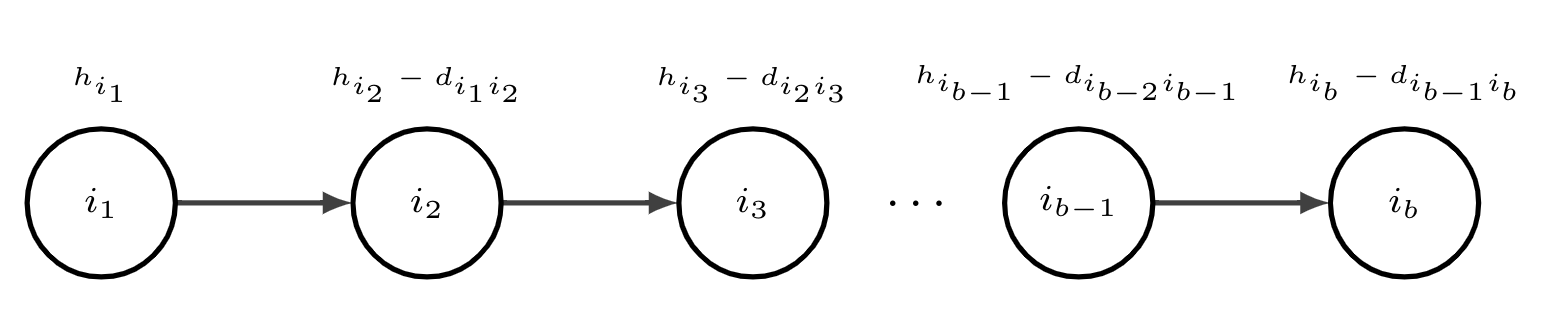}
    \caption{An induced subgraph of a cycle is an one-way path}
    \label{fig:DP}
\end{figure}
Given $b \leq n-1$, for $i \in V(C)$, construct a node set $\overrightarrow{V}_{ib}$ that contains node $i$ with forward arcs connecting nodes starting from $i$ with total number of nodes equals to $b$. Let $i_1, i_2, \ldots, i_{b-1}, i_b$ be the indices in $\overrightarrow{V}_{ib}$. The corresponding set of forward arcs is then $\overrightarrow{C}_{ib} = \{(i_j, i_{j+1}): j \in [1,b-1]\}$ as illustrated in Figure \ref{fig:DP}. We apply the similar logic for backward set of nodes  $\overleftarrow{V}_{ib}$  and arcs $\overleftarrow{C}_{ib}$. Let $F(i,b)$ denote the minimum cost of activating $b$ nodes on the cycle beginning with node $i$. For $1 \leq b \leq n-1$, the minimum cost of activating $b$ nodes on a cycle is given by
\begin{align}
    \min\limits_{i \in V(C)} F(i,b) \label{dp1}
\end{align}
where
\begin{align}
    F(i,b) = \min \left\{ h_i + \sum_{(j,k) \in \overrightarrow{C}_{ib}}(h_k - d_{jk}) , \quad  h_i + \sum_{(j,k) \in \overleftarrow{C}_{ib}}(h_k - d_{jk})\right\}. \label{dp2}
\end{align}

\begin{proposition}
The dynamic programming recursion given by \eqref{dp1} and \eqref{dp2} solves LCIM on a simple cycle in $O(n)$ time. 
\end{proposition}
\begin{proof}
The recursion \eqref{dp2} evaluates the minimum cost of activating $b$ nodes on a path beginning with node $i$ on both directions. As we only need to compare the minimum of \eqref{dp2} for every node, we obtain the optimal objective function of LCIM on a simple cycle in $O(n)$ time. 
\end{proof}

\subsection{Valid inequalities for influence propagation over a cycle}
Since every network consists of trees and cycles as substructures, observe that for every cycle in the network, at least one node is either paid with full incentive or the activation requires influence exertion from nodes outside the cycle. Fischetti et al. \cite{fischetti2018least} first recognized this observation and proposed a generalized propagation constraints in a different space of variables. Here we propose an exponential class of valid inequalities that captures this observation as well as ensures the influence propagation is acyclic for $\conv{(\mathcal{Q})}$. 

\begin{proposition} \label{prop:UC}
Given an inequality \eqref{ucbase} and a cycle with set of nodes $V(C)$ and set of arcs $C$, for $U \subseteq V(C)$, the $(U,C)$ inequality
\begin{align}
    \sum_{i \in U} \gamma_i \left(x_i + \sum_{j \in N(i)} \alpha_{ji} y_{ji} - \beta_i z_i \right) \geq \delta(U) \left(1 - \sum_{(k,\ell) \in C: \ell \notin U}(z_{\ell}-y_{k\ell}) \right)  \label{ucinequ}
\end{align}
is valid  for $\mathcal{Q}$, where 
\begin{itemize}
    \item[(i)] $\omega_i = h_i - \beta_i + \sum_{j \in N(i): j \notin V(C) }(\alpha_{ji} - d_{ji})$,
    \item[(ii)] $\delta(U) = \omega_i$ if $|U| = 1$,
    \item[(iii)] $\delta(U)$ computes the least common multiple of $\omega_i$ for $i \in U$ if $|U| \geq 2$,
    \item[(iv)] $\gamma_i = \frac{\delta(U)}{\omega_i}$.
\end{itemize}
\end{proposition}

\begin{proof}
Due to constraint \eqref{LCIM3}, there exists $z_i = 1$ for some $i \in V(C)$. Let $H = \{(k,\ell) \in C: \ell \notin U\}$. We partition $H$ into two disjoint sets $H_0$ and $H_1$ such that $H=H_0 \cup H_1 $ and $H_0 \cap H_1 = \varnothing $, where $H_0 =\{(k,\ell) \in H: k \notin U\}$ and $H_0 =\{(k,\ell) \in H: k \in  U\}$. We distinguish two main cases:
\begin{case}
\normalfont We first consider $z_i = 0$ for all $i \in U$ and $z_i = 1$ for all $i \in V(C) \setminus U$. 
    In this case, the left hand side of the inequality is equal to 0, whereas in the right hand side, we have $k \in V(C)\setminus U$ and $\ell \in \in V(C)\setminus U$, if there exists at least one $y_{k\ell}=1$, we must have $z_{k}=z_{\ell}=1$, which leads to $|V(C)\setminus U| > |H_0|$. Consequently, 
        \begin{align}
         \delta(U) \left(1 - \sum_{i \in V(C) \setminus U}z_i + \sum_{(k,\ell) \in H_0}y_{k\ell} \right) < 0 \notag
    \end{align}
    holds, inequality \eqref{ucinequ} is thus valid.  
 \end{case}
 \begin{case}
 \normalfont Next we consider $z_i = 1$ for all $i \in U$ and $z_i = 0$ for all $i \in V(C) \setminus U$.
    In this case, we must have $y_{k\ell}=0$ for all $(k,\ell) \in H$ as influence exertion towards nodes belong to $V(C)\setminus U$ is unnecessary. Then
    the inequality reduces to 
    \begin{align}
      \sum_{i \in U} \gamma_i \left(x_i + \sum_{j \in N(i)} \alpha_{ji} y_{ji} - \beta_i\right) \geq \delta(U). \notag 
    \end{align}
    Regardless of whether $y_{ji}=0$ for all $i \in N(i)$ or $y_{ji}=1$ for some $j \in N(i)$, there exists at least one index $i \in U$ such that $x_i = h_i$ and $y_{ji}=0$. In other words, at least one node is activated with full incentive payment in order to launch the propagation for nodes in $U$. For other nodes that receive partial or zero incentives, their corresponding term in the left hand side is zero. To compute the possible range of the left hand side, we rearrange the terms and obtain the following
    \begin{align}
        & \sum_{i \in U} \gamma_i (h_i - \beta_i) \notag\\
       = &\sum_{i \in U} \frac{\delta(U)}{(h_i - \beta_i)} (h_i - \beta_i) \notag\\
       =  &\sum_{i \in U} \delta(U). \notag
    \end{align}
    Consequently, the left hand side is at most $|U|\delta(U)$ and at least $\delta(U)$ with one particular node with full incentive $x_i=h_i$, which is valid. This completes the proof. 
\end{case}
\end{proof}

\begin{exa}
 \normalfont Consider a graph $G$ illustrated in Figure \ref{fig:3cycle}. The number next to each arc is the influence weight $d_{ij}$, while the number inside the brackets next to each node is the threshold $h_i$.   
\begin{figure}[H]
    \centering
    \includegraphics[scale=0.5]{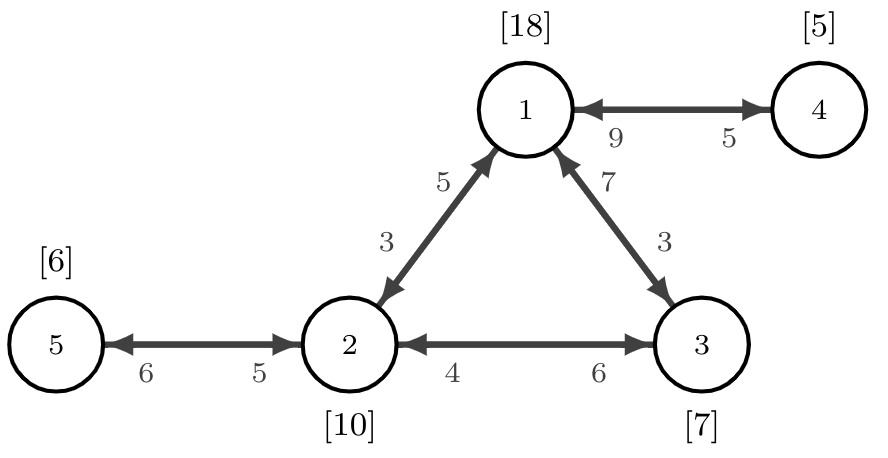}
    \caption{A social network with $n=5$ (Example 2).}
    \label{fig:3cycle}
\end{figure}
For node 2, take $L=\{1,3,5\}$ with $\lambda_2 = 3+4+5-10 = 2$, the corresponding continuous packing inequality \eqref{ucbase} is
\begin{align}
    x_2 + y_{12} + 2y_{32} + 3y_{52} \geq 6z_2. \notag
\end{align}
Similarly, for node 3, take $L=\{1,2\}$ with $\lambda_3 = 6 + 3 -7 = 2$, the corresponding continuous packing inequality \eqref{ucbase} is 
\begin{align}
    x_3 + y_{13} + 4y_{23} \geq 5z_3. \notag
\end{align}
There are two cycles $\{(1,2),(2,3),(3,1)\}$ and $\{(1,3),(3,2),(2,1)\}$ in Figure \ref{fig:3cycle}. For $U = \{2,3\}$ with cycle $C=\{(1,2),(2,3),(3,1)\}$, we have $\omega_2 = 10-6+3-5=2$, $\omega_3 = 7-5=2$, $\delta(\{2,3\})= \texttt{lcm}(2,2)=2$, then the $(U,C)$ inequality is 
\begin{align}
     (x_2 + y_{12} + 2y_{32} + 3y_{52} - 6z_2) + (x_3 + y_{13} + 4y_{23}-5z_3) \geq 2(1-z_1+y_{31}). \notag
\end{align}
Furthermore, consider the cycle $C=\{(1,3),(3,2),(2,1)\}$ from another direction, the $(U,C)$ inequality is 
\begin{align}
     (x_2 + y_{12} + 2y_{32} + 3y_{52} - 6z_2) + (x_3 + y_{13} + 4y_{23}-5z_3) \geq 2(1-z_1+y_{21}). \notag
\end{align}
\end{exa}

Next, we give the condition in which inequality \eqref{ucinequ} is stronger than the generalized cycle elimination constraints. Without loss of generality, we assume that $\delta(U)= 1$ for $U=\varnothing$.

\begin{proposition} \label{prop:dominance}
Inequality \eqref{ucinequ} with $U=\varnothing$ dominates the generalized cycle elimination constraints. 
\end{proposition}

\begin{proof}
Consider a particular cycle $C$, the GCEC can be states as 
\begin{align}
\sum_{(i,j)\in C}(z_j - y_{ij}) \geq z_k,  \notag
\end{align}
where $k \in V(C)$ is an arbitrary choice of index among $V(C)$. For a $(U,C)$ inequality \eqref{ucinequ} with $U = \varnothing$ on this cycle, we obtain
\begin{align}
\sum_{(i,j)\in C}(z_j - y_{ij}) \geq 1.  \notag
\end{align}
Clearly, the GCEC is weaker then the $(U,C)$ inequality unless $z_k^*=1$.
Furthermore, if we have $\sum_{(i,j)\in C}(z_j - y_{ij}) < z_k$, then $\sum_{(i,j)\in C}(z_j - y_{ij}) < 1$ must hold. This implies that there exist violated $(U,C)$ inequalities for every violated cycle $C$ identified. 
\end{proof}

\subsection{Separation of $(U,C)$ inequalities} \label{sec:3.3}

Since the size of of inequalities \eqref{ucinequ} is exponential, we explore a separation scheme to find the most violated inequality corresponding to set $U$ in polynomial time. We again assume that $|V(C)|=n$ in this section.

\begin{proposition}\label{sep_prop}
Separation problem for inequality \eqref{ucinequ} can be solved in $O(n^3 \log n)$ time.
\end{proposition}

\begin{proof}
Inequality \eqref{ucinequ} is violated if
\begin{align}
  \delta(U) \left(1 - \sum_{(k,\ell) \in C: \ell \notin U}(z_{\ell}^* - y_{k\ell}^*) \right) - \sum_{i \in U} \gamma_i \left(x_i^* + \sum_{j \in N(i)} \alpha_{ji} y_{ji}^* - \beta_i z_i^* \right) > 0. \label{sepUC}
\end{align}
For a given fractional point $(\vec{x^*},\vec{y^*},\vec{z^*}) \in \mathcal{Q}$, we determine a set $U \subseteq V(C)$ such that the left hand side of \eqref{sepUC} is maximized. With the observation in Proposition \ref{prop:dominance}, for every violated cycle detected, we construct a longest path problem on a directed acyclic network to solve the separation problem. 

\begin{figure}
    \centering
    \includegraphics[scale=0.45]{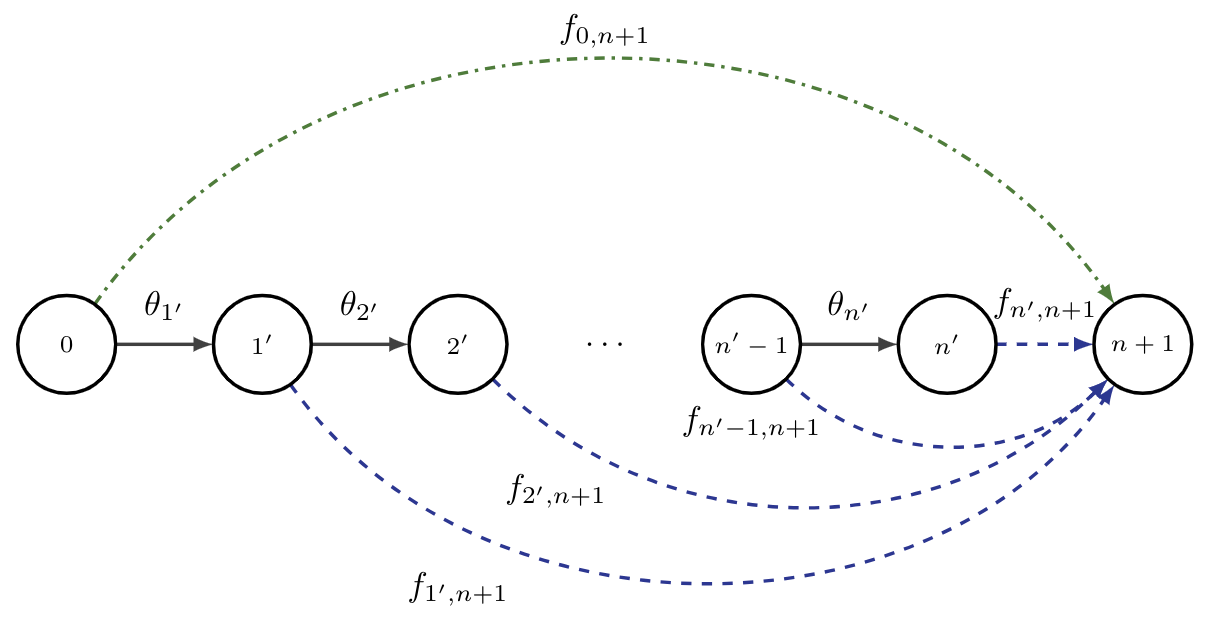}
    \caption{Graph $\mathcal{D}$ for separation of inequality \eqref{ucinequ}}
    \label{fig:sep_uc}
\end{figure}

Consider a directed acyclic network $\mathcal{D} = (\mathcal{V},\mathcal{A})$ with a source vertex $0 \in \mathcal{V}$ and a sink vertex $n + 1 \in \mathcal{V}$. Define $\theta_i = x_i^* + \sum_{j \in N(i)} \alpha_{ji} y_{ji}^* - \beta_i z_i^* $ for all $i \in [1,n]$.  It is possible that there exists multiple inequality \eqref{ucbase} for a fixed $i$. We select the one with the minimum value of $\theta_i$ accordingly. Let index set $\{i^\prime: i \in [1,n]\}$ such that $\theta_{1^\prime} \leq \theta_{2^\prime} \leq \ldots \leq \theta_{n^\prime}$. Node $i^\prime$ is sorted according the the value of $\theta_i$ and each node $i^\prime \in \mathcal{V}$ has a unique mapping to each node $i \in V(C)$. The node set $\mathcal{V}$ is then $\{0,n+1\} \cup \{ i^\prime:  i \in [1,n]\}$.  The arc set is $\mathcal{A} = \{(0,n+1)\} \cup \{(0,1^\prime)\} \cup \{(i^\prime, (i+1)^\prime ) : i \in [1,n-1]\} \cup \{(i^\prime, n+1): i \in [1,n]\}$.

Next, we assign length on each arc in $\mathcal{A}$. For arc $(0,n+1)$, we let
\begin{align}
   f_{0,n+1} = \max \left\{ \omega_i \left(1- \sum_{(k,\ell) \in C: \ell \neq i} (z_{\ell}^*-y_{k\ell}^*) \right) - \theta_i : i \in V(C) \right\}. \notag
\end{align}
Let $f_{0,1^\prime} = \theta_{1^\prime}$. For arcs $\{(i^\prime, (i+1)^\prime ) : i \in [1,n-1]\}$, we set the length $f_{i^\prime,(i+1)^\prime} = \theta_{(i+1)^\prime}$. For arcs $\{(i^\prime, n+1): i \in [1,n]\}$, we set the length
\begin{align}
    & f_{i^\prime,n+1} = \notag \\
    & \delta(\{\omega_j\}_{j=1^\prime}^{i^\prime}) \left(1- \sum_{(k,\ell) \in C: \ell \notin \{j\}_{j=1^\prime}^{i^\prime}} (z_{\ell}^*-y_{k\ell}^*)\right) - \sum_{k=1^\prime}^{i^\prime} \left( \frac{\delta(\{\omega_j\}_{j=1^\prime}^{i^\prime})}{\omega_k}+1\right)\theta_k. \notag
\end{align}
This longest path problem depicted in Figure \ref{fig:sep_uc} can be solved by Dijkstra's algorithm. There exists a violated inequality \eqref{ucinequ} if and only if the longest path is strictly positive and the nodes on this path determine the elements in set $U$. The sorting process of $\theta_i$ takes $O(n\log n)$ time, the evaluation of $f_{0,n+1}$ takes $O(n)$ time, and
the longest path on a directed acyclic graph takes $O(n)$ time as there are $n+2$ nodes and $2n+1$ arcs. Since we have to solve this problem for every violated cycle found by Dijkstra algorithm, the separation algorithm runs $O(n^3\log n)$ time overall.
\end{proof}

\begin{exa} \label{exa:3}
\begin{figure}
    \centering
    \includegraphics[scale=0.5]{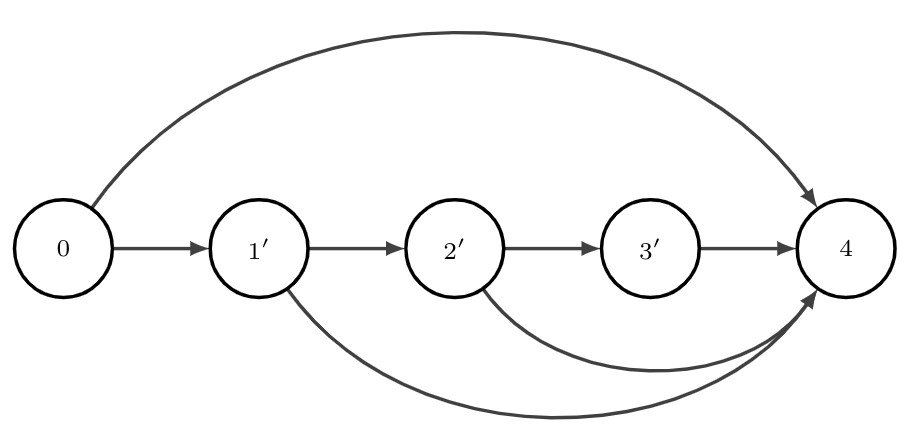}
    \caption{A DAG for separating inequality \eqref{ucinequ} in Example \ref{exa:3}}
    \label{fig:sep_ex}
\end{figure}
\normalfont Consider solving LCIM with $b=3$ on a graph depicted in Figure \ref{fig:3cycle}. The initial linear programming relaxation solution without enumerating any cycle elimination constraints is
\begin{align}
    &(x_1^*,x_2^*,x_3^*, x_4^*,x_5^*) = (0,4.92,0.6,3,0),\notag \\
    &(z_1^*,z_2^*,z_3^*,z_4^*,z_5^*) = (0.6,0.6,0.6,0.6,0.6), \notag \\
    &(y_{12}^*,y_{13}^*, y_{14}^*, y_{21}^*,y_{23}^*, y_{25}^*,y_{31}^*,y_{32}^*,y_{41}^*,y_{52}^*)= (0.36,0,0,0.24,0.6,0.6,0.6,0,0.6,0), \notag
\end{align}
and the objective function value is 8.52. The violated cycle is $\{(1,2),(2,3),(3,1)\}$ for this solution. Then we obtain $\theta_1 = -0.72$, $\theta_2=1.68$ and $\theta_3 = 0$ and sort them in non-decreasing order. Since $\theta_1 < \theta_3 < \theta_2$, the set $\{1^\prime, 2^\prime, 3^\prime\}$ is corresponding to the original node set $\{1,3,2\}$. This leads to a directed acyclic network illustrated in Figure \ref{fig:sep_ex} with new node set $\mathcal{V}=\{0,1^\prime, 2^\prime, 3^\prime, 4\}$. The length of $f_{0,1^\prime}, f_{1^\prime,2^\prime}$ and $f_{2^\prime,3^\prime}$ are $\theta_1, \theta_3$ and $\theta_2$, respectively. The lengths of the remaining arcs are 
\begin{align}
    &f_{0,4} = \max\{3(1-0.24)-(-0.72), 2-1.68, 2\} = 3, \notag\\
    &f_{1^\prime,4} = 3(1-0.24)-(\frac{3}{3}+1)(-0.72) = 3.72, \notag \\
    &f_{2^\prime,4} = 6(1-0.24)-(\frac{6}{3}+1)(-0.72) = 6.72, \notag \\
    &f_{3^\prime,4} = 6-(\frac{6}{3}+1)(-0.72) - (\frac{6}{2}+1)(1.68) = 1.44. \notag
\end{align}
The longest path is $0 \rightarrow 1^\prime \rightarrow 2^\prime \rightarrow 4$, which determines $U=\{1,3\}$ with maximum violation 6. We add the following violated inequality \eqref{ucinequ}
\begin{align}
 2(x_1 + 2y_{21} + 4y_{31} + 6y_{41}-12z_1 )+ 3( x_3 + y_{13} + 4y_{23}-5z_3) \geq 6(1-z_2+y_{12}) \notag
\end{align}
to cut off this fractional solution. We then obtain the new objective function value 10.2 and the solution sets
\begin{align}
    &(x_1^*,x_2^*,x_3^*, x_4^*,x_5^*) = (0,6,2.2,2,0),\notag \\
    &(z_1^*,z_2^*,z_3^*,z_4^*,z_5^*) = (0.6,0.6,0.6,0.6,0.6), \notag \\
    &(y_{12}^*,y_{13}^*, y_{14}^*, y_{21}^*,y_{23}^*, y_{25}^*,y_{31}^*,y_{32}^*,y_{41}^*,y_{52}^*)= (0,0,0.2,0.6,0.6,0.6,0.6,0,0.4,0),\notag
\end{align}
which is very close to the true optimal objective function value 11 in this example. 
\end{exa}

\section{LCIM under equal influence and 100\% adoption} \label{sec:4}
Since the optimal propagation subgraph of LCIM is acyclic, the solution of LCIM on a tree provides a valid lower bound for LCIM on a graph with cycles. Moreover, in practical applications, it is common to assume that both threshold and influence exertion are identical for every node, due to simplicity or lack of accurate estimation. For example, in the unanimous threshold model \cite{chen2009approximability}, $h_i = v_i$ for all $i \in V$ is assumed. This diffusion model is normally considered as the most influence resistant one, and it has applications in complex computer network security problems. In addition, the majority threshold model \cite{valente1996social} assumes $h_i = \lceil \frac{v_i}{2}\rceil$ for all $i \in V$. Both information diffusion models assume that $d_{ij}=1$ for all $(i,j) \in E$. 

In this special case of LCIM where equal influence is assumed for all $i \in V$ and 100\% coverage is required on a tree, we have $d_{ij}=d_i$ for all $i \in V$ and GCEC can be discarded. The LCIM formulation corresponding to equal influence and 100\% coverage on a tree graph is given by
\begin{subequations}
\begin{align}
\text{(LCIM-TE) } \min\limits_{x,y} \quad & \sum_{i \in V} x_i  \notag  \\
\text{s.t.  } & x_i + d_i\sum_{j \in N(i)}y_{ji} \geq h_i \quad \forall i \in V \label{LCIM-E1} \\
& y_{ij}+y_{ji}=1 \quad \forall (i,j) \in E: i < j \label{LCIM-E2} \\
& x \in \mathbb{R}_+^n \label{LCIM-E3} \\
& y \in \mathbb{B}^m. \label{LCIM-E4}
\end{align}
\end{subequations}
Let $\mathcal{S}$ denote the set of feasible solutions to LCIM-TE on a tree graph and let $\mathcal{R}$ denote the set of feasible solutions to the linear programming relaxation of \eqref{LCIM-E1} - \eqref{LCIM-E4}. G{\"u}nne{\c{c}} et al. \cite{gunnecc2020least} prove that LCIM-TE is polynomial solvable on a tree graph. They propose a compact extended formulation with total unimodular constraints by considering three types of incoming influence for every node. However, this extended formulation cannot be applied to unequal influence weights directly, since there would be exponentially many distinct possible types of incoming influence for a node, depending on the influence weights received from its activated neighbors. We give the complete linear description of $\conv{(\mathcal{S})}$ in the original space of variables with additional $O(n)$ constraints and show they are a special case of the continuous cover and continuous packing inequalities by adjusting the influence weights.
\begin{proposition} \label{prop:Econv}
Let $\sigma_i = \lceil \frac{h_i}{d_i} \rceil $  and $g_i = h_i - (\sigma_i -1)d_i$ for all $i \in V$, the inequality 
\begin{align}
    x_i + \min \{g_i,d_i\} \sum_{j \in N(i)}y_{ji} \geq g_i \sigma_i\label{equineq}
\end{align}
is facet-defining for $\conv{(\mathcal{S})}$ if and only if $g_i < d_i$ and $\sigma_i \geq 2$. Furthermore, the complete linear description of $\conv{(\mathcal{S})}$ is given by
\begin{align}
    \conv{(\mathcal{S})} = \left\{(x,y) \in \mathcal{R}: x_i + \min \{g_i,d_i\} \sum_{j \in N(i)}y_{ji} \geq g_i\sigma_i\quad \forall i \in V \right\}. \notag
\end{align}
\end{proposition}

\begin{proof}
 If $g_i=d_i$, then $h_i = \sigma_id_i $ and inequality \eqref{equineq} coincides with $\eqref{LCIM-E1}$. Similarly, if $\sigma_i=1$, we also have $g_i = h_i=d_i$ and inequality \eqref{equineq} is reduced to $\eqref{LCIM-E1}$. To prove the sufficiency, we demonstrate that inequality \eqref{equineq} is a special case of the continuous cover and continuous packing inequalities. Observe that $\sigma_i$ is the minimum number that exceeds $h_i$ if multiplied by $d_i$, which implies that  $h_i - (\sigma_i-1)d_i > 0$. Therefore, $\sigma_i$ is the cardinality of set $L$ corresponding to the continuous packing inequality with equal influence weights. We thus obtain $\lambda_i = d_i\sigma_i - h_i$ and $g_i = d_i - \lambda_i$. Equivalently,
 \begin{align*}
     g_i & = h_i - (\sigma_i-1)d_i \\
         & = h_i + \left[|N(i)|-\sigma_i +1 - |N(i)|\right]d_i \\
         & = \pi_i,
 \end{align*}
 hence $|N(i)|-\sigma_i + 1$ is the cardinality of set $S$ in the continuous cover inequality. Following the result of Lemma \ref{lemma1} for the interchangeable relationship between sets $S$ and $L$, $g_i$ and $ g_i\sigma_i$ coincide with the coefficients of the continuous cover and continuous packing inequalities, respectively. 

For the second part of this Proposition, we assume $g_i < d_i $ and $g_i \sigma_i< h_i$ holds for all $i \in V$ without loss of generality. Observe that for $i \in V$, the possible values of $x_i = \max\{0, h_i - (\sigma_i- w)d_i \}$ for $w \in [0, \sigma_i]$, where $\sigma_i- w$ is an implicit upper bound of number of activated neighbors for node $i$, namely, $\sum_{j \in N(i)}y_{ji} \leq \sigma_i- w$. We prove that for  any choice of $w \in [0, \sigma_i]$, we must have integral $(\vec{x},\vec{y})$ in the following three cases:

\begin{case}
\normalfont Suppose $w=0$ and $x_i=0$. Inequality \eqref{LCIM-E1} is reduced to $\sum_{j \in N(i)}y_{ji} \geq \frac{h_i}{d_i}$, which is dominated by inequality \eqref{equineq} with $\sum_{j \in N(i)}y_{ji} \geq \sigma_i$. There exist at least $\sigma_i$ activated neighbors that exert influence toward node $i$. Moreover, from the implicit upper bound  $\sum_{j \in N(i)}y_{ji} \leq \sigma_i$, we have $\sum_{j \in N(i)}y_{ji} = \sigma_i$. Since $x_i = 0$ we must have $y_{ij}=0$ and $y_{ji}=1$ such that $|\{j: j \in N(i)\}| = \sigma_i$ due to constraints \eqref{LCIM-E2}. 
\end{case}

\begin{case}
\normalfont Suppose $w = \sigma_i$ and $x_i = h_i$. Inequality \eqref{LCIM-E1} becomes $\sum_{j \in N(i)}y_{ji} \geq 0$, which dominates inequality \eqref{equineq} with $\sum_{j \in N(i)}y_{ji} \geq \sigma_i - \frac{h_i}{g_i}$ as the right hand side here is strictly negative. Similar to Case 1, we must have $y_{ji} = 0$ and $y_{ij}=1$ for all $j \in N(i)$ due to the implicit upper bound $\sum_{j \in N(i)}y_{ji} \leq 0$ and constraints \eqref{LCIM-E2}. 
\end{case}

\begin{case}
\normalfont Suppose $w \in [1, \sigma_i -1]$. First, let $w = 1$, then $x_i = g_i = h_i - (\sigma_i-1)d_i$. We have $\sum_{j \in N(i)}y_{ji} \geq \sigma_i - 1$ in both inequality \eqref{LCIM-E1} and inequality \eqref{equineq}. Following Case 1 and Case 2, we have $\sum_{j \in N(i)}y_{ji}=\sigma_i-1$ with $y_{ji}=1$ and $y_{ij}=0$ for some $j \in N(i)$ such that $|\{j: j \in N(i)\}| = \sigma_i-1$. Next, let $w=2$ and $\sum_{j \in N(i)}y_{ji} \geq \sigma_i - 2$ holds in inequality \eqref{LCIM-E1}. While in inequality \eqref{equineq}, 
\begin{align*}
    \sum_{j \in N(i)}y_{ji} & \geq  \frac{g_i\sigma_i - h_i + d_i\sigma_i - 2d_i}{g_i}  \\ 
    & \geq \sigma_i - 1 - \frac{d_i}{g_i}. 
\end{align*}
Since we assume $g_i < d_i$, inequality \eqref{LCIM-E1} dominates inequality \eqref{equineq}. By mathematical induction, for $w \in [1, \sigma_i-1]$, we conclude that inequality \eqref{LCIM-E1} $\sum_{j \in N(i)}y_{ji} \geq w$ always dominates inequality \eqref{equineq}. Furthermore, we must have $\sum_{j \in N(i)}y_{ji} = w$ from the implicit upper bound and the value of $y_{ij}$ and $y_{ji}$ are either 0 or 1 following Case 1 and Case 2. We have now demonstrated that inequality \eqref{equineq} is facet-defining and $(\vec{x},\vec{y})$ are integral for any choice of $w \in [0, \sigma_i]$, thus the proof is completed. 
\end{case}
\end{proof}

We close this section by noting that the $(U,C)$ inequalities \eqref{ucinequ} and the separation algorithm in Section \ref{sec:3.3} can be directly applied to LCIM-TE on a graph with cycles.

\begin{proposition} \label{UC_for_E}
The $(U,C)$ inequality for equal influence weights of a cycle is given by
\begin{align}
    \sum_{i \in U} \gamma_i \left(x_i + \alpha_i \sum_{j \in N(i)} y_{ji} - \beta_i \right) \geq \delta(U) \left(1 -|V(C)|+|U|+\sum_{(k,\ell) \in C: \ell \neq i}y_{k\ell} \right) \label{ucinequ-E}
\end{align}
for $\conv{(\mathcal{S})}$.
\end{proposition}
\begin{proof}
The result is deduced from Proposition \ref{prop:UC} by fixing $z_i$ to 1 for all $i \in V$ and substituting the coefficients accordingly. The definition of $\gamma_i$ and $\delta (U)$ follows Proposition \ref{prop:UC}. Let $\omega_i = h_i - \beta_i + |\{j \in N(i): j \notin V(C)\}|(\alpha_i - d_i)$ for all $i \in V(C)$. The right hand side of the inequality is equivalent to fix $z_i=1$ for $i \notin U$. Finally, let $\alpha_i = \min \{g_i,d_i\}$ and $\beta_i = g_i \sigma_i$ as in inequality \eqref{equineq}. 
\end{proof}

\section{Computational Experiments} \label{sec:5}
In this section, we give a detailed description of the data generation and algorithm settings. We test the effectiveness of a delayed cut generation algorithm that incorporates the proposed valid inequalities in solving LCIM under different conditions. All the experiments were conducted on a single thread of a Windows 10 Enterprise server with Intel(R) Core i7-4770 CPU at 3.40 GHz x-64 based processor and 8GB of RAM using Python 3.8 and Gurobi 9.1.2 with default settings as the optimization solver. A 3600 seconds time limit was imposed for each experiment.

\subsection{Data generation and algorithm settings}
We follow the exact data generation scheme in \cite{fischetti2018least}, except for the fact that we generate bidirectional arcs between every two nodes. The small-world network topology for each instance is generated based on \texttt{watts\_strogatz\_graph} function in the \texttt{NetworkX} package of Python \cite{SciPyProceedings_11}. The instances have the following properties: size of node set $n \in \{50,75,100\}$, average node degree $v \in \{4,8,12,16\}$, rewiring probability $q \in \{0.1, 0.3\}$ and we set penetration rate $a \in \{0.1,0.25,0.5,0.75,1\}$. Influence weight $d_{ij}$ for all $(i,j) \in E$ are generated from discrete uniform distribution between 1 and 10. Let $\Delta_i = \sum_{j \in N(i)}d_{ji}$ and $\Upsilon_i$ be a random variable follows normal distribution $\mathcal{N}(0.7\Delta_i,\Delta_i/v_i)$ for all $i \in V$. We set $h_i = \lceil \max\{1, \min\{\Upsilon_i, \Delta_i\}\} \rceil$. For each setting, we generate three instances and report the average.

The effectiveness of two delayed cut generation algorithms and one alternative reformulation are compared in our study:
\begin{enumerate}
    \item \textbf{DEF}: formulation LCIM given by \eqref{LCIM1} - \eqref{LCIM3},  
    \item \textbf{CB}: formulation LCIM with cut-and-branch enhancement, and 
    \item \textbf{LN}: layered-network formulation.
\end{enumerate}

To implement the delayed cut generation, the GCEC \eqref{LCIM4} is separated via lazy constraint callback only for integer solutions for \textbf{DEF} and \textbf{CB}. Gr{\"o}tschel et al. \cite{grotschel1985acyclic} give a shortest path algorithm for separating \eqref{LCIM4} and we utilize the existing shortest path function \texttt{dijkstra\_path} in the \texttt{NetworkX} package to perform such task. For the cut-and-branch enhancement in algorithm \textbf{CB}, we add the proposed inequalities via user-cut callback at the root node to tighten the linear programming relaxation of formulation LCIM. 

We use the layered-network formulation proposed by \cite{manzour2021integer} to replace \eqref{LCIM4} in algorithm \textbf{LN} as this formulation gives a directed acyclic graph with the additional layer assignment variables $l_i$ for all $i \in V$. We are interested in testing whether this cycle-free formulation is beneficial to solve LCIM compared with GCEC \eqref{LCIM4}. In addition, this formulation allows us to keep the mixed 0-1 knapsack substructure so the valid inequalities can be applied to it directly. In our preliminary experiments, we observe that this formulation does not produce better optimality gap for $a \in \{0.1, 0.25, 0.5, 0.75\}$ by relaxing constraints \eqref{LN-1} compared with \textbf{DEF} and \textbf{CB} (average optimality gap $> 90\%$), hence, we only report the computation for $a = 1$ $(b=n)$. The layered-network formulation used in algorithm \textbf{LN} is given by
\begin{align}
    \min \left\{ \sum_{i \in V}x_i: (x,y,z) \text{ satisfies }  \eqref{LCIM1}, \eqref{LCIM3}, \eqref{LN-1} - \eqref{LN-3} \right\}, \notag
\end{align}
where
\begin{subequations}
\begin{align}
    & y_{ij} + y_{ji} = 1 \quad \forall (i,j) \in E \label{LN-1} \\
    & y_{ji} - (n-1)y_{ij} \leq l_j - l_i \quad \forall (i,j) \in E \label{LN-2} \\
    & 1 \leq l_i \leq n \quad \forall i \in V. \label{LN-3}
\end{align}
\end{subequations}

\subsection{Analysis of results}
\begin{table}[ht!]
\caption{Computational performance comparing MIP nodes, cuts, time and unsolved instances on network with $n$=50. Column \textbf{Time[Gap]*} reports the average solution time (in seconds) of the instances that are solved to optimality, and, where applicable, the average of the optimality gap (\%, in brackets) of the instances that are not solved to optimality when reaching time limit. Each asterisk sign indicates an unsolved instance. }
\label{tab:n50}
\centering
\resizebox{\textwidth}{!}{
\begin{tabular}{ccc|ccccccccc}
\toprule
\multicolumn{3}{c}{$n$=50} & \multicolumn{3}{c}{Nodes} & \multicolumn{3}{c}{Cuts} & \multicolumn{3}{c}{Time{[}Gap{]}*} \\ \toprule
$v-m$ & $q$ & $a$ & DEF & CB & LN & DEF & CB & LN & DEF & CB & LN \\ \hline
\multirow{5}{*}{4-200} & \multirow{5}{*}{0.1} & 0.1 & 18 & 12 &  & 12 & 12 &  & 0.30 & 0.31 &  \\
 &  & 0.25 & 121 & 53 &  & 25 & 15 &  & 0.45 & 0.40 &  \\
 &  & 0.5 & 492 & 254 &  & 52 & 21 &  & 0.84 & 0.65 &  \\
 &  & 0.75 & 1562 & 1680 &  & 76 & 45 &  & 2.53 & 2.74 &  \\
 &  & 1 & 921 & 636 & 1041 & 70 & 27 & 233 & 1.15 & 0.78 & 17.84 \\ \hline
\multirow{5}{*}{4-200} & \multirow{5}{*}{0.3} & 0.1 & 1 & 5 &  & 5 & 4 &  & 0.29 & 0.20 &  \\
 &  & 0.25 & 1 & 1 &  & 13 & 3 &  & 0.33 & 0.27 &  \\
 &  & 0.5 & 33 & 9 &  & 28 & 10 &  & 0.41 & 0.39 &  \\
 &  & 0.75 & 129 & 99 &  & 45 & 18 &  & 0.61 & 0.56 &  \\
 &  & 1 & 207 & 1 & 140 & 47 & 16 & 83 & 0.57 & 0.42 & 5.02 \\ \hline
\multirow{5}{*}{8-400} & \multirow{5}{*}{0.1} & 0.1 & 85 & 115 &  & 18 & 15 &  & 1.02 & 1.71 &  \\
 &  & 0.25 & 429 & 392 &  & 14 & 14 &  & 2.41 & 2.85 &  \\
 &  & 0.5 & 6077 & 8831 &  & 57 & 65 &  & 19.08 & 27.31 &  \\
 &  & 0.75 & 457292 & 196845 &  & 189 & 169 &  & 1367.06{[}7.31{]}* & 682.35 &  \\
 &  & 1 & 838587 & 658888 & 69506 & 339 & 341 & 2491 & {[}8.22{]}*** & {[}4.70{]}*** & {[}8.56{]}*** \\ \hline
\multirow{5}{*}{8-400} & \multirow{5}{*}{0.3} & 0.1 & 104 & 85 &  & 19 & 14 &  & 0.96 & 1.49 &  \\
 &  & 0.25 & 375 & 497 &  & 17 & 18 &  & 2.50 & 3.81 &  \\
 &  & 0.5 & 3009 & 3592 &  & 35 & 40 &  & 15.88 & 31.34 &  \\
 &  & 0.75 & 323777 & 185520 &  & 148 & 170 &  & 2028.51{[}5.11{]}* & 1486.88 &  \\
 &  & 1 & 418483 & 404013 & 57946 & 297 & 265 & 2372 & {[}10.47{]}*** & {[}7.13{]}*** & {[}12.08{]}*** \\ \bottomrule
\end{tabular}}
\end{table}

\begin{table}[ht!]
\caption{Computational performance comparing MIP nodes, cuts, time and unsolved instances on network with $n$=75. Column \textbf{Time[Gap]*} reports the average solution time (in seconds) of the instances that are solved to optimality, and, where applicable, the average of the optimality gap (\%, in brackets) of the instances that are not solved to optimality when reaching time limit. Each asterisk sign indicates an unsolved instance. }
\label{tab:n75}
\centering
\resizebox{\textwidth}{!}{
\begin{tabular}{cccccccccccc}
\toprule
\multicolumn{3}{c}{$n$=75} & \multicolumn{3}{c}{Nodes} & \multicolumn{3}{c}{Cuts} & \multicolumn{3}{c}{Time{[}Gap{]}*} \\ \toprule
$v-m$ & $q$ & \multicolumn{1}{c|}{$a$} & DEF & CB & LN & DEF & CB & LN & DEF & CB & LN \\ \hline
\multirow{5}{*}{4-300} & \multirow{5}{*}{0.1} & \multicolumn{1}{c|}{0.1} & 91 & 17 &  & 11 & 8 &  & 0.57 & 0.31 &  \\
 &  & \multicolumn{1}{c|}{0.25} & 833 & 742 &  & 46 & 22 &  & 1.87 & 1.35 &  \\
 &  & \multicolumn{1}{c|}{0.5} & 1094 & 1543 &  & 82 & 48 &  & 2.11 & 3.17 &  \\
 &  & \multicolumn{1}{c|}{0.75} & 3378 & 1880 &  & 119 & 69 &  & 5.92 & 4.17 &  \\
 &  & \multicolumn{1}{c|}{1} & 1753 & 1384 & 1252 & 130 & 69 & 366 & 2.08 & 2.15 & 18.39 \\ \hline
\multirow{5}{*}{4-300} & \multirow{5}{*}{0.3} & \multicolumn{1}{c|}{0.1} & 10 & 12 &  & 17 & 9 &  & 0.64 & 0.70 &  \\
 &  & \multicolumn{1}{c|}{0.25} & 309 & 76 &  & 34 & 13 &  & 0.85 & 0.71 &  \\
 &  & \multicolumn{1}{c|}{0.5} & 805 & 151 &  & 78 & 21 &  & 1.78 & 1.24 &  \\
 &  & \multicolumn{1}{c|}{0.75} & 1537 & 604 &  & 80 & 37 &  & 3.50 & 2.66 &  \\
 &  & \multicolumn{1}{c|}{1} & 976 & 965 & 1067 & 107 & 53 & 445 & 2.44 & 2.12 & 16.63 \\ \hline
\multirow{5}{*}{8-600} & \multirow{5}{*}{0.1} & \multicolumn{1}{c|}{0.1} & 229 & 195 &  & 13 & 10 &  & 2.73 & 3.45 &  \\
 &  & \multicolumn{1}{c|}{0.25} & 3178 & 912 &  & 26 & 20 &  & 18.82 & 7.81 &  \\
 &  & \multicolumn{1}{c|}{0.5} & 216230 & 212500 &  & 192 & 78 &  & 227.75{[}4.77{]}* & 857.06 &  \\
 &  & \multicolumn{1}{c|}{0.75} & 564387 & 348760 &  & 248 & 169 &  & {[}6.15{]}*** & 508.18{[}3.36{]}** &  \\
 &  & \multicolumn{1}{c|}{1} & 360594 & 379776 & 50255 & 368 & 324 & 2923 & {[}9.50{]}*** & {[}5.71{]}*** & {[}9.64{]}*** \\ \hline
\multirow{5}{*}{8-600} & \multirow{5}{*}{0.3} & \multicolumn{1}{c|}{0.1} & 214 & 127 &  & 13 & 5 &  & 2.10 & 2.78 &  \\
 &  & \multicolumn{1}{c|}{0.25} & 1078 & 1224 &  & 19 & 15 &  & 10.54 & 10.27 &  \\
 &  & \multicolumn{1}{c|}{0.5} & 9948 & 3776 &  & 68 & 75 &  & 129.29 & 113.47 &  \\
 &  & \multicolumn{1}{c|}{0.75} & 161508 & 223027 &  & 184 & 194 &  & {[}5.00{]}*** & 3489.72{[}1.99{]}** &  \\
 &  & \multicolumn{1}{c|}{1} & 191592 & 167297 & 41342 & 266 & 356 & 2701 & {[}9.48{]}*** & {[}6.38{]}*** & {[}10.30{]}*** \\ \hline
\multicolumn{1}{l}{\multirow{5}{*}{12-900}} & \multirow{5}{*}{0.1} & \multicolumn{1}{c|}{0.1} & 812 & 635 &  & 9 & 10 &  & 13.93 & 18.00 &  \\
\multicolumn{1}{l}{} &  & \multicolumn{1}{c|}{0.25} & 1803 & 2376 &  & 18 & 13 &  & 53.28 & 62.91 &  \\
\multicolumn{1}{l}{} &  & \multicolumn{1}{c|}{0.5} & 369964 & 229309 &  & 36 & 63 &  & 2803.78{[}2.11{]}* & {[}2.96{]}*** &  \\
\multicolumn{1}{l}{} &  & \multicolumn{1}{c|}{0.75} & 18614 & 31354 &  & 111 & 144 &  & {[}13.68{]}*** & {[}10.28{]}*** &  \\
\multicolumn{1}{l}{} &  & \multicolumn{1}{c|}{1} & 48188 & 33974 & 26375 & 228 & 230 & 3317 & {[}16.49{]}*** & {[}16.18{]}*** & {[}17.85{]}*** \\ \hline
\multicolumn{1}{l}{\multirow{5}{*}{12-900}} & \multirow{5}{*}{0.3} & \multicolumn{1}{c|}{0.1} & 1163 & 1275 &  & 10 & 2 &  & 15.08 & 30.70 &  \\
\multicolumn{1}{l}{} &  & \multicolumn{1}{c|}{0.25} & 3757 & 5019 &  & 42 & 19 &  & 115.19 & 224.71 &  \\
\multicolumn{1}{l}{} &  & \multicolumn{1}{c|}{0.5} & 27217 & 22714 &  & 35 & 48 &  & 1363.66{[}2.83{]}* & 1827.81{[}9.39{]}* &  \\
\multicolumn{1}{l}{} &  & \multicolumn{1}{c|}{0.75} & 11042 & 12625 &  & 58 & 120 &  & {[}24.75{]}*** & {[}15.69{]}*** &  \\
\multicolumn{1}{l}{} &  & \multicolumn{1}{c|}{1} & 20954 & 25899 & 23253 & 199 & 283 & 3071 & {[}22.71{]}*** & {[}17.21{]}*** & {[}26.88{]}*** \\ \bottomrule
\end{tabular}}
\end{table}

We summarize our computational results in TABLE \ref{tab:n50}, \ref{tab:n75} and \ref{tab:n100} under various settings of $(n,m,v,q,a)$. We report the average number of branch-and-cut tree nodes explored in the column \textbf{Nodes}. The column \textbf{Cuts} shows the average number of Gurobi cuts added during the optimization process. In column \textbf{Time[Gap]*} we report the average solution time (in seconds) of the instances that are solved to optimality, and the average of the optimality gap of the instances that are not solved to optimality when reaching time limit (in brackets). Each asterisk sign indicates an unsolved instance and the gap is calculated by 100 $\times (\texttt{ub}-\texttt{lb})/\texttt{lb}$ where \texttt{ub} and \texttt{lb} are the best integer feasible solution obtained and best lower bound generated by the algorithm within time limit, respectively. 

We observe that the major factor that contributes to the unsolved instances with positive optimality gap is the average node degree rather than the number of nodes. This observation can be justified by comparing the instances $(n,v,m)$ = (100,4,400) and $(n,v,m)$ = (50,8,400), as the former is easier to solve than the later. A similar observation is also established in \cite{fischetti2018least} in their set covering formulation using the price-cut-and-branch algorithm. Despite the average number of nodes and cuts are the greatest in \textbf{LN} for $n$=50, there exists no clear domination relationship in columns \textbf{Nodes} and \textbf{Cuts} between \textbf{DEF}, \textbf{CB} and \textbf{LN} for $n$=75 or 100. Moreover, the average number of cuts added is not very large, which indicates that Gurobi cuts do not complement the optimization process. For the unsolved instances in Table \ref{tab:n50} and \ref{tab:n75}, \textbf{CB} outperforms \textbf{DEF} and \textbf{LN} except for instances $(n,v,m,q,a)$ = (75,12,900,0.1,0.5) and (75,12,900,0.3,0.5). In Table \ref{tab:n100} where $n$=100, \textbf{CB} still outperforms \textbf{DEF} and \textbf{LN} for $v \in$\{4,8,12\} except for one instance (100,12,1200,0.3,0.5), where the optimality gap difference is 0.92\%. The common setting $a$=0.5 shared in these exceptions suggests that the symmetry created by the cardinality constraint requires additional improvement. We also notice that \textbf{LN} produces better optimality gap than \textbf{DEF} for instances (100,8,800,0.1,1) and (100,12,1200,0.1,1). However, \textbf{LN} suffers from slow improvement of both upper and lower bound and results in larger ending gap in general in our experiments. We begin to see the degraded performance of algorithm \textbf{CB} for instances with $v$=16. A large number of valid inequalities added to the root node could be a possible reason that decelerates the optimization process. Five out of ten results generated by \textbf{CB} under this category are no better than \textbf{DEF}. Nevertheless, the average optimality gap produced by \textbf{CB} for these five instances are 1.49\% higher than \textbf{DEF}. To sum up, although the linear programming relaxation of the arc-based formulation is very weak with zero objective function value in most cases, our proposed valid inequalities significantly improve the strength of the lower bound and effectively reduce or close the optimality gap.

\begin{table}[H]
\caption{Computational performance comparing MIP nodes, cuts, time and unsolved instances on network with $n$=100. Column \textbf{Time[Gap]*} reports the average solution time (in seconds) of the instances that are solved to optimality, and, where applicable, the average of the optimality gap (\%, in brackets) of the instances that are not solved to optimality when reaching time limit. Each asterisk sign indicates an unsolved instance. }
\label{tab:n100}
\centering
\resizebox{\textwidth}{!}{
\begin{tabular}{cccccccccccc}
\toprule
\multicolumn{3}{c}{$n$=100} & \multicolumn{3}{c}{Nodes} & \multicolumn{3}{c}{Cuts} & \multicolumn{3}{c}{Time{[}Gap{]}*} \\ \toprule
$v-m$ & $q$ & \multicolumn{1}{c|}{$a$} & DEF & CB & LN & DEF & CB & LN & DEF & CB & LN \\ \hline
\multirow{5}{*}{4-400} & \multirow{5}{*}{0.1} & \multicolumn{1}{c|}{0.1} & 28 & 1 &  & 11 & 4 &  & 0.95 & 0.85 &  \\
 &  & \multicolumn{1}{c|}{0.25} & 148 & 118 &  & 50 & 12 &  & 1.24 & 1.09 &  \\
 &  & \multicolumn{1}{c|}{0.5} & 1554 & 998 &  & 80 & 33 &  & 3.77 & 2.49 &  \\
 &  & \multicolumn{1}{c|}{0.75} & 1705 & 834 &  & 128 & 61 &  & 6.58 & 3.62 &  \\
 &  & \multicolumn{1}{c|}{1} & 669 & 1600 & 670 & 120 & 63 & 318 & 3.09 & 2.67 & 13.01 \\ \hline
\multirow{5}{*}{4-400} & \multirow{5}{*}{0.3} & \multicolumn{1}{c|}{0.1} & 15 & 11 &  & 15 & 3 &  & 0.53 & 0.36 &  \\
 &  & \multicolumn{1}{c|}{0.25} & 83 & 123 &  & 22 & 7 &  & 1.27 & 1.07 &  \\
 &  & \multicolumn{1}{c|}{0.5} & 365 & 391 &  & 92 & 26 &  & 2.40 & 1.60 &  \\
 &  & \multicolumn{1}{c|}{0.75} & 2045 & 697 &  & 117 & 53 &  & 8.55 & 2.92 &  \\
 &  & \multicolumn{1}{c|}{1} & 904 & 213 & 994 & 124 & 42 & 383 & 3.84 & 1.52 & 18.65 \\ \hline
\multirow{5}{*}{8-800} & \multirow{5}{*}{0.1} & \multicolumn{1}{c|}{0.1} & 511 & 295 &  & 17 & 7 &  & 7.11 & 7.28 &  \\
 &  & \multicolumn{1}{c|}{0.25} & 16138 & 9611 &  & 16 & 59 &  & 96.09 & 88.48 &  \\
 &  & \multicolumn{1}{c|}{0.5} & 364834 & 278716 &  & 297 & 157 &  & 1499.18{[}4.20{]}** & 1247.35{[}2.37{]}* &  \\
 &  & \multicolumn{1}{c|}{0.75} & 175014 & 231825 &  & 282 & 306 &  & {[}9.81{]}*** & {[}7.10{]}*** &  \\
 &  & \multicolumn{1}{c|}{1} & 121123 & 154716 & 46686 & 374 & 439 & 2988 & {[}13.84{]}*** & {[}8.39{]}*** & {[}13.02{]}*** \\ \hline
\multirow{5}{*}{8-800} & \multirow{5}{*}{0.3} & \multicolumn{1}{c|}{0.1} & 434 & 345 &  & 16 & 12 &  & 5.74 & 7.27 &  \\
 &  & \multicolumn{1}{c|}{0.25} & 5682 & 4509 &  & 19 & 21 &  & 56.08 & 72.36 &  \\
 &  & \multicolumn{1}{c|}{0.5} & 65328 & 25567 &  & 91 & 129 &  & 735.63{[}1.79{]}* & 685.49 &  \\
 &  & \multicolumn{1}{c|}{0.75} & 13119 & 37109 &  & 190 & 233 &  & {[}14.13{]}*** & {[}6.43{]}*** &  \\
 &  & \multicolumn{1}{c|}{1} & 40550 & 68921 & 35379 & 261 & 407 & 2849 & {[}15.19{]}*** & {[}9.39{]}*** & {[}17.02{]}*** \\ \hline
\multicolumn{1}{l}{\multirow{5}{*}{12-1200}} & \multicolumn{1}{l}{\multirow{5}{*}{0.1}} & \multicolumn{1}{c|}{0.1} & 1866 & 1249 &  & 6 & 4 &  & 44.77 & 83.37 &  \\
\multicolumn{1}{l}{} & \multicolumn{1}{l}{} & \multicolumn{1}{c|}{0.25} & 6425 & 6356 &  & 29 & 48 &  & 296.98 & 372.08 &  \\
\multicolumn{1}{l}{} & \multicolumn{1}{l}{} & \multicolumn{1}{c|}{0.5} & 40341 & 41264 &  & 41 & 52 &  & {[}11.17{]}*** & {[}9.17{]}*** &  \\
\multicolumn{1}{l}{} & \multicolumn{1}{l}{} & \multicolumn{1}{c|}{0.75} & 8292 & 7915 &  & 143 & 173 &  & {[}19.69{]}*** & {[}16.40{]}*** &  \\
\multicolumn{1}{l}{} & \multicolumn{1}{l}{} & \multicolumn{1}{c|}{1} & 15506 & 16755 & 20873 & 187 & 388 & 4447 & {[}29.78{]}*** & {[}18.11{]}*** & {[}20.74{]}*** \\ \hline
\multicolumn{1}{l}{\multirow{5}{*}{12-1200}} & \multicolumn{1}{l}{\multirow{5}{*}{0.3}} & \multicolumn{1}{c|}{0.1} & 3984 & 2870 &  & 0 & 9 &  & 39.06 & 145.17 &  \\
\multicolumn{1}{l}{} & \multicolumn{1}{l}{} & \multicolumn{1}{c|}{0.25} & 55111 & 33660 &  & 14 & 65 &  & 1187.31 & 2191.00 &  \\
\multicolumn{1}{l}{} & \multicolumn{1}{l}{} & \multicolumn{1}{c|}{0.5} & 15457 & 11751 &  & 44 & 51 &  & {[}21.00{]}*** & {[}21.92{]}*** &  \\
\multicolumn{1}{l}{} & \multicolumn{1}{l}{} & \multicolumn{1}{c|}{0.75} & 4426 & 5168 &  & 102 & 128 &  & {[}30.63{]}*** & {[}24.46{]}*** &  \\
\multicolumn{1}{l}{} & \multicolumn{1}{l}{} & \multicolumn{1}{c|}{1} & 8704 & 10036 & 20128 & 174 & 241 & 3688 & {[}21.11{]}*** & {[}19.05{]}*** & {[}21.31{]}*** \\ \hline
\multicolumn{1}{l}{\multirow{5}{*}{16-1600}} & \multicolumn{1}{l}{\multirow{5}{*}{0.1}} & \multicolumn{1}{c|}{0.1} & 4585 & 5403 &  & 14 & 11 &  & 193.29 & 474.52 &  \\
\multicolumn{1}{l}{} & \multicolumn{1}{l}{} & \multicolumn{1}{c|}{0.25} & 9590 & 11593 &  & 26 & 28 &  & 1244.07 & 2010.94 &  \\
\multicolumn{1}{l}{} & \multicolumn{1}{l}{} & \multicolumn{1}{c|}{0.5} & 5384 & 4268 &  & 9 & 14 &  & {[}22.36{]}*** & {[}22.42{]}*** &  \\
\multicolumn{1}{l}{} & \multicolumn{1}{l}{} & \multicolumn{1}{c|}{0.75} & 3360 & 3591 &  & 68 & 92 &  & {[}32.71{]}*** & {[}30.71{]}*** &  \\
\multicolumn{1}{l}{} & \multicolumn{1}{l}{} & \multicolumn{1}{c|}{1} & 8464 & 6772 & 11907 & 148 & 158 & 3443 & {[}22.47{]}*** & {[}23.43{]}*** & {[}24.10{]}*** \\ \hline
\multicolumn{1}{l}{\multirow{5}{*}{16-1600}} & \multicolumn{1}{l}{\multirow{5}{*}{0.3}} & \multicolumn{1}{c|}{0.1} & 11982 & 13063 &  & 5 & 20 &  & 286.70 & 852.37 &  \\
\multicolumn{1}{l}{} & \multicolumn{1}{l}{} & \multicolumn{1}{c|}{0.25} & 51785 & 23015 &  & 19 & 33 &  & {[}19.46{]}*** & {[}22.45{]}*** &  \\
\multicolumn{1}{l}{} & \multicolumn{1}{l}{} & \multicolumn{1}{c|}{0.5} & 5271 & 4520 &  & 17 & 18 &  & {[}34.36{]}*** & {[}35.58{]}*** &  \\
\multicolumn{1}{l}{} & \multicolumn{1}{l}{} & \multicolumn{1}{c|}{0.75} & 2636 & 2724 &  & 85 & 89 &  & {[}42.24{]}*** & {[}44.46{]}*** &  \\
\multicolumn{1}{l}{} & \multicolumn{1}{l}{} & \multicolumn{1}{c|}{1} & 4441 & 4485 & 11802 & 78 & 91 & 3408 & {[}25.78{]}*** & {[}24.78{]}*** & {[}25.79{]}*** \\ \bottomrule
\end{tabular}}
\end{table}

\section{Conclusion} \label{sec:6}
We study the least cost influence maximization problem in social networks where the influence propagation behavior among users is captured by the deterministic linear threshold model. A typical application of this problem is to obtain an estimation of the partial incentives given to early product adopters in viral marketing while achieving a desired coverage rate by the end of information spreading. We focus on the case where influence weights exerted from peers are heterogeneous and derive several classes of valid inequalities from the hidden mixed 0-1 knapsack substructure in the mixed-integer programming formulation. Despite the fact that the set of feasible solutions is hard to convexify due to the knapsack constraints and the linear programming relaxation being very weak, our computational experiments show that the delayed cut generation algorithm exploiting these inequalities can effectively reduce the optimality gap. For the case with equal influence weights and 100\% adoption on a tree, we characterize the complete linear description of the convex hull in its natural space of incentive, arc propagation and activation variables. The convex hull of the LCIM with equal influence weights and arbitrary adoption on a tree is still an open question and requires further investigation. We observe that the bottleneck of computational improvements are mostly in the instances where 100\% adoption is not required. A promising future research direction is to identify more explicit valid inequalities from the intersections of the cardinality constraint that controls the penetration rate with the rest of the formulation. 

\section*{Acknowledgements}
We are grateful to the two anonymous reviewers who provided constructive feedback that helped us improve the content of this paper. We also thank Demetrios Papazaharias in the early discussion of the separation problem for the minimum influencing subset inequalities. This material is based upon work supported by the Air Force Research Laboratory (AFRL) Mathematical Modeling and Optimization Institute and AFRL award FA8651-16-2-0009. Preliminary results of this study were presented at the 9th International Conference on Computational Data and Social Networks (CSoNet 2020, December 11--13, 2020, Dallas, TX) and published in the respective conference proceedings volume \cite{chen2020cutting}.

\section*{Data Availability Statement}
The datasets analyzed during the current study are available from the corresponding author on reasonable request.

\bibliographystyle{unsrt}
\bibliography{CBP_2022}
\end{document}